\documentclass[12pt]{amsart}
\usepackage{amscd,amssymb,cite,amsmath,latexsym,comment,epic,eepic,euscript}
\usepackage{graphicx}
\usepackage{enumerate}
\usepackage{float}
\usepackage{standalone}
\usepackage[initials]{amsrefs}
\usepackage[all]{xy}
\usepackage{nomencl}
\usepackage{tikz-cd}
\usepackage{tikz}
\usepackage{subcaption}
\usepackage{xcolor}
\usetikzlibrary{angles,quotes}
\usetikzlibrary{decorations.markings,arrows, decorations.pathreplacing}
\definecolor{mintgreen}{RGB}{152,255,152}
\definecolor{pinksalmon}{RGB}{255,102,102}
\definecolor{hueso}{RGB}{245,245,220}
\definecolor{marfil}{RGB}{255,253,208}
\definecolor{amarillo}{RGB}{255,255,0}
\definecolor{canaryyellow}{rgb}{1.0, 0.94, 0.0}
\usetikzlibrary{decorations.markings,arrows}
\usetikzlibrary{decorations.pathreplacing}
\usetikzlibrary{shapes.geometric,positioning}
\theoremstyle{plain}
\newtheorem{thm}{Theorem}[section]
\newtheorem{lemma}[thm]{Lemma} %%Delete [thm] to re-start numbering
\newtheorem{prop}[thm]{Proposition}

\theoremstyle{remark}

\newtheorem{remark}[thm]{Remark}

\theoremstyle{definition}

\usepackage{amscd,amssymb,comment,epic,eepic,euscript}
\usepackage{graphicx}
\usepackage{enumerate}
\usepackage[initials]{amsrefs}
\usepackage[all]{xy}
\newcommand{\beq}{\begin{equation}}
\newcommand{\eeq}{\end{equation}}
\newcommand{\beqn}{\begin{equation*}}
\newcommand{\eeqn}{\end{equation*}}
\newcommand{\bre}{\begin{remark}}
\newcommand{\ere}{\end{remark}}
\newcommand{\beqno}[1]{\begin{equation}\label{#1}}

\newcommand{\mref}[1]{(\ref{#1})}
\newcommand{\veck}{{\mathbf{k}}}
\newcommand{\Gk}{\Gamma^{\veck}}

\begin{document}

\title%[Degenerate Band Edges]
{Degenerate band edges in periodic quantum graphs}

\author[Berkolaiko]{Gregory Berkolaiko}
\address{G.B., Department of Mathematics, Texas A\&M University,
College Station, TX 77843-3368, USA}
\email{Gregory.Berkolaiko@math.tamu.edu}

\author[Kha]{Minh Kha}
\address{M.K., Department of Mathematics, The University of Arizona,
Tucson, AZ 85721-0089, USA}
\email{minhkha@math.arizona.edu}

\begin{abstract}
  Edges of bands of continuous spectrum of periodic structures arise
  as maxima and minima of the dispersion relation of their
  Floquet--Bloch transform.  It is often assumed that the extrema
  generating the band edges are non-degenerate.

  This paper constructs a family of examples of $\mathbb{Z}^3$-periodic quantum
  graphs where the non-degeneracy assumption fails: the maximum of the
  first band is achieved along an algebraic curve of co-dimension 2.
  The example is robust with respect to perturbations of edge lengths,
  vertex conditions and edge potentials.  The simple idea behind the
  construction allows generalizations to more complicated graphs and
  lattice dimensions.  The curves along which extrema are achieved
  have a natural interpretation as moduli spaces of planar polygons.
\end{abstract}
\maketitle

%%%%%%%%%%%%%%%%%%%%%%%%%%%%%%%%%%%%%%%%%%%%%%%%%%%%%%%%%%%%%%%%%%%%%%%%
\section{Introduction}

Periodic media play a prominent role in many fields including
mathematical physics and material sciences. A classical instance is
the study of crystals, one of the most stable form of all solids that
can be found throughout nature. In a perfectly ordered crystal, the
atoms are placed in a periodic order and this order is responsible for
many properties particular to this material.  On the mathematical level,
the stationary Schr\"odinger operator $-\Delta+V$ with a periodic
potential $V$ is used to describe the one-electron model of solid
state physics \cite{SSP}; here $V$ represents the field created by the
lattice of ions in the crystal.  The resulting differential operator
with periodic coefficients has been studied intensively in mathematics
and physics literature for almost a century.  A standard technique in
spectral analysis of periodic operators is called the Floquet-Bloch
theory (see e.g., \cites{Kbook, Ksurvey}). This technique is applicable
not only to the above model example of periodic Schr\"odinger
operators on Euclidean space, but also to a wide variety of elliptic
periodic equations on manifolds and branching structures (graphs).
Periodic elliptic operators of mathematical physics as well as their
periodic elliptic counterparts on manifolds and quantum graphs do
share an important feature of their spectra: the so-called
\textit{band-gap structure} (see e.g.,
\cites{BK,KOS,Kbook,Ksurvey}). Namely, the spectrum of a periodic
elliptic operator can be represented in a natural way as the union of
finite closed intervals, called \textit{spectral bands}, and sometimes
they may leave open intervals between them, called \textit{spectral
  gaps}.  An endpoint of a spectral gap is called a \textit{gap edge}.
For each spectral band, there is also a corresponding \textit{band
  function} whose image is exactly that spectral band. The set
consisting of all graphs of band functions is called the
\textit{dispersion relation}.  The analytical and geometrical
properties of dispersion relations encode significant information
about the spectral features of the operator.\footnote{These features
  are also called ``threshold effects'' \cite{BirSus2000} whenever
  they depend only on the infinitesimal structure (e.g., a finite
  number of Taylor coefficients) of the dispersion relation at the
  spectral edges.}  Hence studying structural properties of the
dispersion relation may reveal interesting results for periodic
differential operators.  A well-known and widely believed conjecture
in physics literature says that generically (with respect to
perturbations of the coefficients of the operator) the extrema are
attained by a single band of the dispersion relation, are isolated,
and have non-degenerate Hessian.  The non-degeneracy of extrema at the
edges of the spectrum is often assumed to establish many important
results such as finding asymptotics of Green's functions of a periodic
elliptic operator near and at its gap edge \cites{Kha, KKR, KucRai},
homogenization \cites{Bir, BirSus2000, BirSus2004}, or counting
dimensions of spaces of solutions with polynomial growth \cites{KP1,
  KP2}, just to name a few.

In the continuous situation, the generic simplicity of spectral gap
edges was obtained in \cite{KR}. The well-known result in \cite{KS}
established the validity of the full conjecture for the bottom of the
spectrum of a periodic Schr\"{o}dinger operator in Euclidean spaces,
however the full conjecture still remains unproven for internal edges.
It is worth mentioning that in the two dimensional situation, a
``variable period" version of the non-degeneracy conjecture was found
in \cite{ParSht} and the isolated nature of extrema for a wide class
of $\mathbb{Z}^2$-periodic elliptic operators was recently established
in \cite{FilKach}.  In the discrete graph situation, the statement of
the conjecture fails for periodic Schr\"{o}dinger operators on a
diatomic lattice (see \cite{FilKach}). However, in the example of
\cite{FilKach} there are only 2 free parameters to perturb the
operator with and therefore the degeneracy may be attributed to the
paucity of available perturbations.  To investigate this question
further, \cite{DKS} considered a wider class of $\mathbb{Z}^2$-periodic
discrete graphs and it was found that the set of parameters of vertex
and edge weights for which the dispersion relation of the discrete
Laplace-Beltrami operator has a degenerate extremum is a
semi-algebraic subset of co-dimension $1$ in the space of all
parameters.  These examples show that the non-degeneracy of gap edges
is a delicate issue even in the discrete setting.

In this paper, we propose two examples of periodic \emph{metric} (or
``quantum'') graphs whose Schr\"odinger operator dispersion relation
has a degenerate band edge.  Remarkably, this band edge remains
degenerate under a continuum of perturbations: one may vary edge
lengths, vertex coupling constants and the edge potentials.  Our
examples can be considered quantum-graph versions of the
counterexample in \cite{FilKach}, and they clearly show that the main
reason for the degeneracy is not the small number of perturbation
degrees of freedom, but rather the drastic effect a suitably chosen
rank-1 perturbation has on the topology of the graph.

%%%%%%%%%%%%%%%%%%%%%%%%%%%%%%%%%%%%%%%%%%%%%%%%%%%%%%%%%%%%%%%%%%%%%%%% 
\section{The main result}

We now introduce the quantum graph of our main theorem and formulate
the result.  The description of principal notions used in the main
theorem, such as quantum graphs, covers and periodicity, and the
Floquet--Bloch transform, are deferred to Sections~\ref{sec:qg_vc},
%\ref{sec:cover} 
and \ref{sec:fb} correspondingly.  Expanded versions
of these descriptions are available in several sources, such as
\cites{BK,Ber_incol17,Sunada,Ksurvey}.

We will in fact describe two variants of our graph, $X_1$ and $X_2$;
the main theorem will apply equally to both.  We start by describing
one layer of the graph, which looks like planar hexagonal lattice
shown in Figure \ref{fig:SingleLayer}.  It has vertices of two types,
type $A$ and type $B$ denoted by red filled and blue empty circles
correspondingly. 
The graph $X_1$ will have $\delta$-type conditions at
vertices $A$ and $B$, with real coupling constants $\gamma_A$ and
$\gamma_B$, $\gamma_A \neq \gamma_B$.  The graph $X_2$ will have only
Neumann--Kirchhoff (NK) conditions but the vertices of type $A$ are
decorated by attaching a ``tail'', i.e.\ an edge leading to a vertex
of degree one, shown as a smaller black circle in
Figure~\ref{fig:SingleLayer}(right).  Either version is a
$\mathbb{Z}^2$-periodic graph in $\mathbb{R}^2$ and its period lattice
is generated by the two brown dashed vectors.  The edges of the same
color (parallel edges) are related by $\mathbb{Z}^2$-shifts.  They are
assumed to have the same length and to have the same potential (if
any) placed on them.

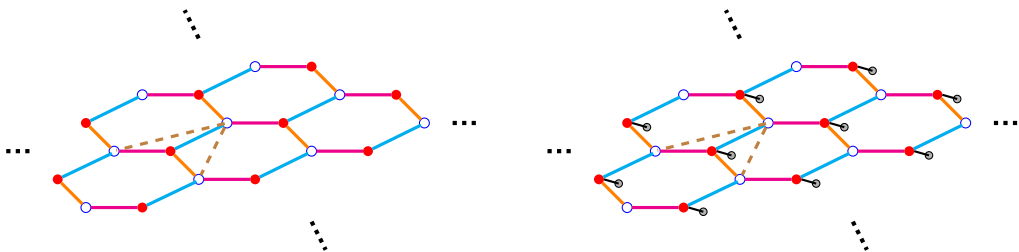
\begin{figure}[H]
  \begin{center}
    \begin{tikzpicture}
[scale = 0.75,
            > = Straight Barb,
phasor/.style = {very thick,-{Stealth}},
angles/.style = {draw, <->, angle eccentricity=1,
                 right, angle radius=7mm}                       ]  
%%%%%%%%%%%%%%%%%%%%%%%%%%%%%%%%%%% Draw nodes of the first layer
\foreach \i in {0,...,2}{
\draw[draw = blue, fill = blue!0,
                           inner sep = 0pt] (\i/2,\i) circle (2.5 pt);
\fill[red] (\i/2+1,\i) circle (2.5 pt);
\draw[draw = blue, fill = blue!0,
                           inner sep = 0pt] (\i/2+2,\i+0.5) circle (2.5 pt);
\fill[red] (\i/2+3,\i+0.5) circle (2.5 pt);
} 
\foreach \i in {0,...,1}{     
\fill[red] (\i/2-0.5,\i+0.5) circle (2.5 pt);
\draw[draw = blue, fill = blue!0,
                           inner sep = 0pt] (\i/2+4,\i+1) circle (2.5 pt);
\fill[red] (\i/2+5,\i+1) circle (2.5 pt);
}            
\draw[draw = blue, fill = blue!0,
                           inner sep = 0pt] (6,1.5) circle (2.5 pt);
%%%%%%%%%%%%%%%%%%%%%%%%%%%%%%%%%%%%%% Connecting blue and red vertices, magenta edges
%% convert pt to cm. 2.5 pt = 0.088cm
\foreach \i in {0,...,2}{
\draw[very thick, color=magenta] (\i/2+0.088,\i) -- (\i/2+0.912,\i);
\draw[very thick, color=magenta] (\i/2+2.088,\i+0.5) -- (\i/2+2.912,\i+0.5);
}
\foreach \i in {0,...,1}{  
\draw[very thick, color=magenta] (\i/2+4.088,\i+1) -- (\i/2+4.912,\i+1);
}
%%%%%%%%%%%%%%%%%%%%%%%%%%%%%%%%%%%%%% Connecting blue and red vertices, orange edges
%%%%% incremental lengths (to avoid overlapping on nodes) in x-coordinates is plus or minus Radius (2.5pt) / sqrt{1+k^2} where k is the slope of the line.
%%%%%  incremental lengths in y-coordinates is plus or minus Radius (2.5pt) * k / sqrt{1+k^2}.
%%%%% orange edges's slopes are k = -1. x-increment = 0.062
\foreach \i in {0,...,3}{
\draw[very thick, color=orange] (2*\i - 0.062,\i/2 + 0.062)--(2*\i-0.438,\i/2+0.438);
}
\foreach \i in {0,...,2}{  
\draw[very thick, color=orange] (2*\i+0.438,\i/2+1.062)--(2*\i+0.062,\i/2+1.438);
}
%%%%%%%%%%%%%%%%%%%%%%%%%%%%%%%%%%%%%% Connecting blue and red vertices, cyan edges
%%%%% cyan edges's slopes are k = 0.5
%%%%% incremental length is 0.078 for x-coordinates
\foreach \i in {0,...,2}{
\draw[very thick, color=cyan] (1.5*\i+0.078,1.539-\i/2)--(1.5*\i+0.938,1.961-\i/2);
\draw[very thick, color=cyan] (1.5*\i+2.078,2.039-\i/2)--(2.938+1.5*\i,2.461-\i/2);
}
\foreach \i in {0,...,1}{  
\draw[very thick, color=cyan] (1.5*\i-0.422,0.539-\i/2)--(1.5*\i+0.461,0.961-\i/2);
}
%%%%%%%%%%%%%%%%%%%%%%%%%%%%%% Draw tails at red vertices on the first layer: shift red vertex by (0.35,-0.075)
%\foreach \i in {0,...,2}{ 
%\fill[black] (\i/2+1.35,\i-0.075) circle (2 pt);
%\draw[thick] (\i/2+1,\i)--(\i/2+1.35,\i-0.075); 
%\fill[black] (\i/2+3.35,\i+0.425) circle (2 pt);
%\draw[very thick] (\i/2+3,\i+0.5)--(\i/2+3.35,\i+0.425);
%} 
%\foreach \i in {0,...,1}{     
%\fill[black] (\i/2-0.15,\i+0.425) circle (2pt);
%\draw[thick] (\i/2-0.5,\i+0.5)--(\i/2-0.15,\i+0.425);
%\fill[black] (\i/2+5.35,\i+0.925) circle (2pt);
%\draw[thick] (\i/2+5,\i+1)--(\i/2+5.35,\i+0.925);
%}
%%%%%%%%%%%%%%%%%%%%%%%%%%%%%% \ldots
\draw[ultra thick, dotted] (-1,1)--(-1.5,1);
\draw[ultra thick, dotted] (6.5,1.5)--(7,1.5);
\draw[ultra thick, dotted] (2,3)--(1.75,3.5);
\draw[ultra thick, dotted] (4,-0.25)--(4.25,-0.75);
%%%%%%%%%%%%%%%%%%%%%%%%%%%%%% Draw periodic vectors
  \draw [very thick, dashed,color=brown] (2.461,1.49) -- (0.538,1.015);
 \draw [very thick, dashed,color=brown] (2.46,1.421) -- (2.039,0.578);
\end{tikzpicture}
    \hspace{10pt}
    \begin{tikzpicture}
[scale = 0.75,
            > = Straight Barb,
phasor/.style = {very thick,-{Stealth}},
angles/.style = {draw, <->, angle eccentricity=1,
                 right, angle radius=7mm}                       ]  
%%%%%%%%%%%%%%%%%%%%%%%%%%%%%%%%%%% Draw nodes of the first layer
\foreach \i in {0,...,2}{
\draw[draw = blue, fill = blue!0,
                           inner sep = 0pt] (\i/2,\i) circle (2.5 pt);
\fill[red] (\i/2+1,\i) circle (2.5 pt);
%\fill[blue] (\i/2+2,\i+0.5) circle (2.5 pt);
\draw[draw = blue, fill = blue!0,
                           inner sep = 0pt] (\i/2+2,\i+0.5) circle (2.5 pt);
\fill[red] (\i/2+3,\i+0.5) circle (2.5 pt);
} 
\foreach \i in {0,...,1}{     
\fill[red] (\i/2-0.5,\i+0.5) circle (2.5 pt);
\draw[draw = blue, fill = blue!0,
                           inner sep = 0pt] (\i/2+4,\i+1) circle (2.5 pt);
\fill[red] (\i/2+5,\i+1) circle (2.5 pt);
}            
\draw[draw = blue, fill = blue!0,
                           inner sep = 0pt] (6,1.5) circle (2.5 pt);
%%%%%%%%%%%%%%%%%%%%%%%%%%%%%%%%%%%%%% Connecting blue and red vertices, magenta edges
%% convert pt to cm. 2.5 pt = 0.088cm
\foreach \i in {0,...,2}{
\draw[very thick, color=magenta] (\i/2+0.088,\i) -- (\i/2+0.912,\i);
\draw[very thick, color=magenta] (\i/2+2.088,\i+0.5) -- (\i/2+2.912,\i+0.5);
}
\foreach \i in {0,...,1}{  
\draw[very thick, color=magenta] (\i/2+4.088,\i+1) -- (\i/2+4.912,\i+1);
}
%%%%%%%%%%%%%%%%%%%%%%%%%%%%%%%%%%%%%% Connecting blue and red vertices, orange edges
%%%%% incremental lengths (to avoid overlapping on nodes) in x-coordinates is plus or minus Radius (2.5pt) / sqrt{1+k^2} where k is the slope of the line.
%%%%%  incremental lengths in y-coordinates is plus or minus Radius (2.5pt) * k / sqrt{1+k^2}.
%%%%% orange edges's slopes are k = -1. x-increment = 0.062
\foreach \i in {0,...,3}{
\draw[very thick, color=orange] (2*\i - 0.062,\i/2 + 0.062)--(2*\i-0.438,\i/2+0.438);
}
\foreach \i in {0,...,2}{  
\draw[very thick, color=orange] (2*\i+0.438,\i/2+1.062)--(2*\i+0.062,\i/2+1.438);
}
%%%%%%%%%%%%%%%%%%%%%%%%%%%%%%%%%%%%%% Connecting blue and red vertices, cyan edges
%%%%% cyan edges's slopes are k = 0.5
%%%%% incremental length is 0.078 for x-coordinates
\foreach \i in {0,...,2}{
\draw[very thick, color=cyan] (1.5*\i+0.078,1.539-\i/2)--(1.5*\i+0.938,1.961-\i/2);
\draw[very thick, color=cyan] (1.5*\i+2.078,2.039-\i/2)--(2.938+1.5*\i,2.461-\i/2);
}
\foreach \i in {0,...,1}{  
\draw[very thick, color=cyan] (1.5*\i-0.422,0.539-\i/2)--(1.5*\i+0.461,0.961-\i/2);
}
%%%%%%%%%%%%%%%%%%%%%%%%%%%%%% Draw tails at red vertices on the first layer: shift red vertex by (0.35,-0.075)
\foreach \i in {0,...,2}{ 
\node [circle,draw=black,fill=black!30,inner sep=1pt,minimum size=2pt] at (\i/2+1.35,\i-0.075) {};
\draw[thick] (\i/2+1.088,\i)--(\i/2+1.35,\i-0.075); 
\node [circle,draw=black,fill=black!30,inner sep=1pt,minimum size=2pt] at (\i/2+3.35,\i+0.425) {};
\draw[thick] (\i/2+3.088,\i+0.5)--(\i/2+3.35,\i+0.425);
} 
\foreach \i in {0,...,1}{     
\node [circle,draw=black,fill=black!30,inner sep=1pt,minimum size=2pt] at (\i/2-0.15,\i+0.425) {};
\draw[thick] (\i/2-0.412,\i+0.5)--(\i/2-0.15,\i+0.425);
\node [circle,draw=black,fill=black!30,inner sep=1pt,minimum size=2pt] at (\i/2+5.35,\i+0.925) {};
\draw[thick] (\i/2+5.088,\i+1)--(\i/2+5.35,\i+0.925);
}
%%%%%%%%%%%%%%%%%%%%%%%%%%%%%% \ldots
\draw[ultra thick, dotted] (-1,1)--(-1.5,1);
\draw[ultra thick, dotted] (6.5,1.5)--(7,1.5);
\draw[ultra thick, dotted] (2,3)--(1.75,3.5);
\draw[ultra thick, dotted] (4,-0.25)--(4.25,-0.75);
%%%%%%%%%%%%%%%%%%%%%%%%%%%%%% Draw periodic vectors
  \draw [very thick, dashed,color=brown] (2.461,1.49) -- (0.538,1.015);
 \draw [very thick, dashed,color=brown] (2.46,1.421) -- (2.039,0.578);
\end{tikzpicture}
  \end{center}
  \caption{Two layers of graphs $X_1$ (left) and $X_2$ (right) respectively. These layers are $\mathbb{Z}^2$-periodic with respect to the
    Bravais lattice generated by the two brown dashed. The only difference between these two layers is the extra black tails added in the right layer.}
  \label{fig:SingleLayer}
\end{figure}

\begin{figure}
  \begin{center}
    \begin{tikzpicture}
%%%%%%%%%%%%%%%%%%%%%%%%%%%%%%%%%%% Draw nodes of the first layer
\foreach \i in {0,...,2}{
\draw[draw = blue, fill = blue!0,
                           inner sep = 0pt] (\i/2,\i) circle (2.5 pt);
\fill[red] (\i/2+1,\i) circle (2.5 pt);
\draw[draw = blue, fill = blue!0,
                           inner sep = 0pt] (\i/2+2,\i+0.5) circle (2.5 pt);
\fill[red] (\i/2+3,\i+0.5) circle (2.5 pt);
} 
\foreach \i in {0,...,1}{     
\fill[red] (\i/2-0.5,\i+0.5) circle (2.5 pt);
\draw[draw = blue, fill = blue!0,
                           inner sep = 0pt] (\i/2+4,\i+1) circle (2.5 pt);
\fill[red] (\i/2+5,\i+1) circle (2.5 pt);
}            
\draw[draw = blue, fill = blue!0,
                           inner sep = 0pt] (6,1.5) circle (2.5 pt);

%%%%%%%%%%%%%%%%%%%%%%%%%%%%%%%%%%% Draw nodes of the second layer, shifting by (0.5,-3)
\foreach \i in {0,...,2}{
\draw[draw = blue, fill = blue!0,
                           inner sep = 0pt] (\i/2+.5,\i-3) circle (2.5 pt);
\fill[red] (\i/2+1.5,\i-3) circle (2.5 pt);
\draw[draw = blue, fill = blue!0,
                           inner sep = 0pt] (\i/2+2.5,\i-2.5) circle (2.5 pt);
\fill[red] (\i/2+3.5,\i-2.5) circle (2.5 pt);
} 
\foreach \i in {0,...,1}{     
\draw[draw = blue, fill = blue!0,
                           inner sep = 0pt] (\i/2-1,\i-2.5) circle (2.5 pt);
\fill[red] (\i/2,\i-2.5) circle (2.5 pt);  
\draw[draw = blue, fill = blue!0,
                           inner sep = 0pt] (\i/2+4.5,\i-2) circle (2.5 pt); 
\fill[red] (\i/2+5.5,\i-2) circle (2.5 pt);
}            
\node[circle,minimum width=4pt,inner sep=0pt]  (6.5,-1.5) at  (6.5,-1.5) {};  
\draw[draw = blue, fill = blue!0,
                           inner sep = 0pt] (6.5,-1.5) circle (2.5 pt);
\fill[red] (-1.5,-2) circle (2.5 pt);
%%%%%%%%%%%%%%%%%%%%%%%%%%%%%%%%%%% Draw nodes of the third layer, shifting by (0.5,-3)
\foreach \i in {0,...,2}{
\draw[draw = blue, fill = blue!0,
                           inner sep = 0pt] (\i/2+1,\i-6) circle (2.5 pt);
\fill[red] (\i/2+2,\i-6) circle (2.5 pt);
\draw[draw = blue, fill = blue!0,
                           inner sep = 0pt] (\i/2+3,\i-5.5) circle (2.5 pt);
\fill[red] (\i/2+4,\i-5.5) circle (2.5 pt);
} 
\foreach \i in {0,...,1}{ 
\draw[draw = blue, fill = blue!0,
                           inner sep = 0pt] (\i/2-0.5,\i-5.5) circle (2.5 pt);     
\fill[red] (\i/2+0.5,\i-5.5) circle (2.5 pt);  
\draw[draw = blue, fill = blue!0,
                           inner sep = 0pt] (\i/2+5,\i-5) circle (2.5 pt); 
\fill[red] (\i/2+6,\i-5) circle (2.5 pt);
}            
\draw[draw = blue, fill = blue!0,
                           inner sep = 0pt] (7,-4.5) circle (2.5 pt);
\fill[red] (-1,-5) circle (2.5 pt);
%%%%%%%%%%%%%%%%%%%%%%%%%%%%%%%%%%%%%% Connecting blue and red vertices, magenta edges
%% 1st layer
\foreach \i in {0,...,2}{
\draw[very thick, color=magenta] (\i/2+0.088,\i) -- (\i/2+0.912,\i);
\draw[very thick, color=magenta] (\i/2+2.088,\i+0.5) -- (\i/2+2.912,\i+0.5);
}
\foreach \i in {0,...,1}{  
\draw[very thick, color=magenta] (\i/2+4.088,\i+1) -- (\i/2+4.912,\i+1);
}
%% 2nd layer
\foreach \i in {0,...,2}{
\draw[very thick, color=magenta] (\i/2+.588,\i-3) -- (\i/2+1.412,\i-3);
\draw[very thick, color=magenta] (\i/2+2.588,\i-2.5) -- (\i/2+3.412,\i-2.5);
}
\foreach \i in {0,...,1}{  
\draw[very thick, color=magenta] (\i/2-0.912,\i-2.5) -- (\i/2-0.088,\i-2.5);
\draw[very thick, color=magenta] (\i/2+4.588,\i-2) -- (\i/2+5.412,\i-2);
}
%% 3rd layer
\foreach \i in {0,...,2}{
\draw[very thick, color=magenta] (\i/2+1.088,\i-6) -- (\i/2+1.912,\i-6);
\draw[very thick, color=magenta] (\i/2+3.088,\i-5.5) -- (\i/2+3.912,\i-5.5);
}
\foreach \i in {0,...,1}{  
\draw[very thick, color=magenta] (\i/2-0.412,\i-5.5) -- (\i/2+0.412,\i-5.5);
\draw[very thick, color=magenta] (\i/2+5.088,\i-5) -- (\i/2+5.912,\i-5);
}
%%%%%%%%%%%%%%%%%%%%%%%%%%%%%%%%%%%%%% Connecting blue and red vertices, orange edges
%% all layers, 4 orange edges
\foreach \i in {0,...,3}{
\draw[very thick, color=orange] (2*\i - 0.062,\i/2 + 0.062)--(2*\i-0.438,\i/2+0.438);
\draw[very thick, color=orange] (2*\i+0.438,\i/2-2.938)--(2*\i+0.062,\i/2-2.562);
\draw[very thick, color=orange] (2*\i+.938,\i/2-5.938)--(2*\i+0.562,\i/2-5.562);
}
\foreach \i in {-1,...,2}{
\draw[very thick, color=orange] (2*\i+1-.062,\i/2-2+.062)--(2*\i+0.5+.062,\i/2-1.5-0.062);
}
\foreach \i in {-1,...,2}{
\draw[very thick, color=orange] (2*\i+1.5-.062,\i/2-5+.062)--(2*\i+1+.062,\i/2-4.5-.062);
}
%% all layers, 3 orange edges
\foreach \i in {0,...,2}{  
\draw[very thick, color=orange] (2*\i+0.5-0.062,\i/2+1+0.062)--(2*\i+0.062,\i/2+1.5-0.062);
\draw[very thick, color=orange] (2*\i+1.5-0.062,\i/2-5+0.062)--(2*\i+1+0.062,\i/2-4.5-0.062);
}
%%%%%%%%%%%%%%%%%%%%%%%%%%%%%%%%%%%%%% Connecting blue and red vertices, cyan edges
%% 1st layer 
\foreach \i in {0,...,2}{
\draw[very thick, color=cyan] (1.5*\i+.078,1.5-\i/2+.039)--(1.5*\i+1-.078,2-\i/2-.039);
\draw[very thick, color=cyan] (1.5*\i+2+.078,2-\i/2+.039)--(3+1.5*\i-.078,2.5-\i/2-.039);
%% 2nd layer 
\draw[very thick, color=cyan] (1.5*\i+0.5+.078,-1.5-\i/2+.039)--(1.5*\i+1.5-.078, -1-\i/2-.039);
\draw[very thick, color=cyan] (1.5*\i+2.5+.078,-1-\i/2+.039)--(3.5+1.5*\i-.078,-0.5-\i/2-.039);
%% 3rd layer
\draw[very thick, color=cyan] (1.5*\i+1+.078,-4.5-\i/2+.039)--(1.5*\i+2-.078, -4-\i/2-.039);
\draw[very thick, color=cyan] (1.5*\i+3+.078,-4-\i/2+.039)--(4+1.5*\i-.078,-3.5-\i/2-.039);
}
%% 1st layers
\foreach \i in {0,...,1}{  
\draw[very thick, color=cyan] (1.5*\i-0.5+0.078,0.5-\i/2+0.039)--(1.5*\i+0.5-.078,1-\i/2-.039);
}
%% 2nd layer
\foreach \i in {-1,...,1}{
\draw[very thick, color=cyan] (1.5*\i+.078,-2.5-\i/2+0.039)--(1.5*\i+1-0.078,-2-\i/2-.039);
}
%% 3rd layer
\foreach \i in {-1,...,1}{
\draw[very thick, color=cyan] (1.5*\i+0.5+.078,-5.5-\i/2+0.039)--(1.5*\i+1.5-0.078,-5-\i/2-.039);
}
%%%%%%%%%%%%%%%%%%%%%%%%%%%%%% Connecting green edges from red on 1st layer to blue on 2nd layer, shift by (-0.5,-3)
\foreach \i in {0,...,2}{
\draw[thin, color=green] (2*\i+1-.014,\i/2-.086)--(2*\i+0.5+.014,\i/2-3+.086);
\draw[thin, color=green] (2*\i-0.5-.014,\i/2+0.5-.086)--(2*\i-1+.014,\i/2-2.5+.086);
\draw[thin, color=green] (5.5-.014,2-.086)--(5+.014,-1+.086);
\draw[thin, color=green] (2*\i-.014,\i/2+1.5-.086)--(2*\i-0.5+.014,\i/2-1.5+.086);
}
%%%%%%%%%%%%%%%%%%%%%%%%%%%%%% Connecting blue and red vertices, green edges from red on 2nd layer to blue on 3rd layer
\foreach \i in {0,...,2}{
\draw[thin, color=green] (\i/2+1+.014,\i-6+.086)--(\i/2+1.5-.014,\i-3-.086);
\draw[thin, color=green]  (\i/2+3+.014,\i-5.5+.086)-- (\i/2+3.5-.014,\i-2.5-.086);
}
\foreach \i in {0,1}{
\draw[thin, color=green] (\i/2+5+.014,\i-5+.086) -- (\i/2+5.5-.014,\i-2-.086);
\draw[thin, color=green] (\i/2-0.5+.014,\i-5.5+.086) -- (\i/2-.014,\i-2.5-.086);
}
\draw[thin, color=green] (-1.5-.014,-2-.086) -- (-1.5625+.014,-2.375+.086);
%%%%%%%%%%%%%%%%%%%%%%%%%%%%%% Some green edges from red vertices on 3rd layer to the next layer 
\foreach \i in {0,...,2}{
\draw[thin, color=green]  (\i/2+2-.014,\i-6-.086) -- (\i/2+1.9375,\i-6.375);
\draw[thin, color=green] (\i/2+4-.014,\i-5.5-.086) -- (\i/2+3.9375,\i-5.875);
} 
\foreach \i in {0,...,1}{     
\draw[thin,color=green] (\i/2+0.5-.014,\i-5.5-.086) -- (\i/2+0.4375,\i-5.875);
\draw[thin, color=green] (\i/2+6-.014,\i-5-.086) -- (\i/2+5.9375,\i-5.375);
}            
\draw[thin,color=green] (-1-.014,-5-.086) -- (-1.0625,-5.375);
\draw[ultra thick, dotted] (-1,1)--(-1.5,1);
\draw[ultra thick, dotted] (6.5,1.5)--(7,1.5);
%% 2nd layer, left and right
\draw[ultra thick, dotted] (-2,-2)--(-2.5,-2);
\draw[ultra thick, dotted] (7,-1.5)--(7.5,-1.5);
\draw[ultra thick, dotted] (-1.6125,-3)--(-1.675,-3.375);
%% 3rd layer, left and right
\draw[ultra thick, dotted] (-1.5,-5)--(-2,-5);
\draw[ultra thick, dotted] (7.5,-4.5)--(8,-4.5);
%% top
\draw[ultra thick, dotted] (1.5,3.25)--(1.375,2.5); %right top
\draw[ultra thick, dotted] (5,3.25)--(4.875,2.5); %left top
%% bottom 
\draw[ultra thick, dotted] (6.1,-5.75) -- (5.975,-6.5); % right bottom
\draw[ultra thick, dotted] (3,-6.15)--(2.875,-6.9); % middle bottom
\draw[ultra thick, dotted] (-1.1,-6)--(-1.225,-6.75); % left bottom
\end{tikzpicture}
  \end{center}
  \caption{The graph $X_1$ is generated by stacking together infinitely
    many copies of the layer graph along the height axis. A layer is
    connected to the next layer by certain green edges. To get the graph $X_2$, one just simply adds black tails at the red filled vertices of $X_1$.}
%    \GB{I suggest
%      changing this picture to $X_1$ because it will have less stuff;
%      leave $X_2$ to reader's imagination.  Change the style of the
%      vertices; green edges may be shown in a different thickness (see
%      if it works well).}
  \label{fig:MultiLayer}
\end{figure}
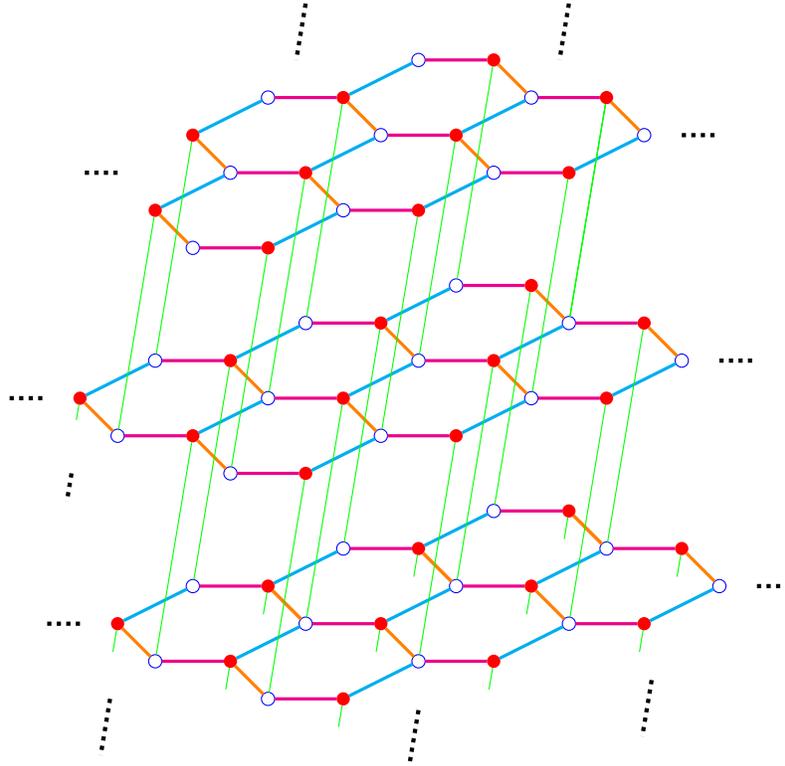

The $\mathbb{Z}^3$-periodic graphs $X_1$ and $X_2$ are obtained by
stacking the corresponding layers infinitely many times in both
directions of the height axis, see Figure~\ref{fig:MultiLayer}.  The
layers are connected in a periodic fashion by edges (shown in green)
between vertices of type B in a lower level and vertices of type A in
the upper level.  Roughly speaking, one may think of the result as an
infinite sheeted cover of the layers in Figure~\ref{fig:SingleLayer}.
In particular, $X_1$ is a $3$-dimensional topological diamond lattice,
see \cite{Sunada}.  In Figure~\ref{fig:FundDomain} we sketch a choice
of the fundamental domain of the graph $X_1$ with respect to the
$\mathbb{Z}^3$-periodic lattice.

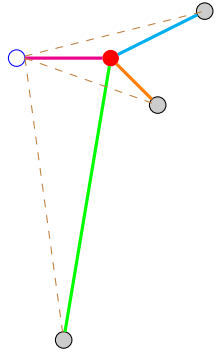
\begin{figure}[h]
  \begin{center}
    \begin{tikzpicture}
[scale=1.25,
            > = Straight Barb,
phasor/.style = {very thick,-{Stealth}},
angles/.style = {draw, <->, angle eccentricity=1,
                 right, angle radius=7mm}                       ]  
%%%%%%%%%%%%%%%%%%%%%%%%%%%%%%%%%%% First layer
\foreach \i in {1}{
\fill[red] (\i/2+1,\i) circle (2.5 pt);
}         
%\node [circle,draw,inner sep=2pt,minimum size=2.5pt] at (0,0) {};
\draw[draw = blue, fill = blue!0,
                           inner sep = 0pt] (0.5,1) circle(2.5 pt);
%% Grey vertices on first layer                           
\draw[draw = black, fill = black!20,
                           inner sep = 0pt] (2,0.5) circle(2.5 pt);
\draw[draw = black, fill = black!20,
                           inner sep = 0pt] (2.5,1.5) circle(2.5 pt);

%%%%%%%%%%%%%%%%%%%%%%%%%%%%%%%%%%%% Second layer
\draw[draw = black, fill = black!20,
                           inner sep = 0pt] (1,-2) circle(2.5 pt);
%\node [circle,draw,inner sep=2pt,minimum size=2.5pt] at (2,-2) {};
%\node [circle,draw,inner sep=2pt,minimum size=2.5pt] at (0.5,-1.5) {};
%\node [circle,draw,inner sep=2pt,minimum size=2.5pt] at (0,-2.5) {};
% %%%%%%%%%%%%%%%%%%%%%%%%%%%%%%%%% magenta edge for second layer
%\draw[very thick, color=magenta] (1+0.088,-2) -- (1+0.912,-2);              
% %%%%%%%%%%%%%%%%%%%%%%%%%%%%%%%%%%%%%% orange edge for 2nd layer
%\draw[very thick, color=orange] (1 - 0.062, -2 + 0.062)--(1-0.438,-2+0.438);
%%%%%%%%%%%%%%%%%%%%%%%%%%%%%%%%%%%%%%% cyan edge for 2nd layer
%\draw[very thick, color=cyan] (1-.078,-2-.039)--(.078, -2.5+.039);
                           
%%%%%%%%%%%%%%%%%%%%%%%%%%%%%%%%%%%%%% Connecting blue and red vertices, magenta edges
\foreach \i in {1}{
\draw[very thick, color=magenta] (\i/2+0.088,\i) -- (\i/2+0.912,\i);
}
%%%%%%%%%%%%%%%%%%%%%%%%%%%%%%%%%%%%%% Connecting blue and red vertices, orange edges
\foreach \i in {1}{
\draw[very thick, color=orange] (2*\i - 0.062,\i/2 + 0.062)--(2*\i-0.438,\i/2+0.438);
}
%%%%%%%%%%%%%%%%%%%%%%%%%%%%%%%%%%%%%% Connecting blue and red vertices, cyan edges
\foreach \i in {1}{
\draw[very thick, color=cyan] (1.5*\i+.078,1.5-\i/2+.039)--(1.5*\i+1-.078,2-\i/2-.039);
}
%%%%%%%%%%%%%%%%%%%%%%%%%%%%%%%%%%%%%% Green edges
\foreach \i in {1}{
\draw[very thick, color=green] (2*\i-0.5-.014,\i/2+0.5-.086)--(2*\i-1+.014,\i/2-2.5+.086);
}

%%%%%%%%%%%%%%%%%%%%%%%%%%%%%% Draw tails at red vertices on the first layer: shift red vertex by (0.35,-0.075)
%\foreach \i in {1}{ 
%\fill[black] (\i/2+1.35,\i-0.075) circle (2 pt);
%\draw[thick] (\i/2+1,\i)--(\i/2+1.35,\i-0.075); 
%} 
%%%%%%%%%%%%%%%%%%%%%%%%%%%%%%% Draw periodic vectors
% \draw [very thick, dashed,color=brown] (0,0) -- (0.5,1);
% \draw [very thick, dashed,color=brown] (0,0) -- (2,0.5);
 \draw [thin, dashed,color=brown] (0.5 + .085,1+.021) -- (2.5,1.5);
 \draw [thin, dashed,color=brown] (0.5+.088,1) -- (2,0.5);
 \draw [thin, dashed,color=brown] (0.5+.088,1) -- (1,-2);
%  \draw [very thick, dashed,color=brown] (2,-2) -- (0,-2.5);
% \draw [very thick, dashed,color=brown] (0.5,-1.5) -- (0,-2.5);
% \draw [->,phasor,very thick, dashed,color=brown] (2.5,1.5) -- (0.5,1);
%\draw [->,phasor,very thick, dashed,color=brown] (2.5,1.5) -- (2,0.5);
\end{tikzpicture}
  \end{center}
  \caption{A fundamental domain for the graph $X_1$. Here the three gray vertices are not included in the fundamental domain. The graph $X_1$ can be obtained by shifting this fundamental domain along the three dashed directions, which are its periods.}
  \label{fig:FundDomain}
\end{figure}

The graphs $X_1$ and $X_2$ we defined above are actually the
\textit{maximal abelian covers} of finite graphs (see e.g., \cites{Baez, Sunada} for more details on maximal abelian covers of graphs).  Taking the quotient
of $X$ with respect to the periodic lattice we obtain the respective
graphs in Figure~\ref{fig:CompactGraph}.  The graph
$\Gamma_1 = X_1 / \mathbb{Z}^3$ has two vertices, $A$ and $B$, which
are the images of the vertices of type $A$ and $B$ in $X_1$ under the
canonical covering map from $X_1$ to $\Gamma_1$.  The four edges of
$\Gamma_1$ are the images of the sets of parallel edges in $X_1$.  The
graph $\Gamma_2 = X_2 / \mathbb{Z}^3$ has three vertices and five
edges.  For either graph $\Gamma$, the first integral homology group
is $H_1(\Gamma,\mathbb{Z})\cong\mathbb{Z}^3$.  We will be using
notation $X$ when a statement applies equally to both $X_1$ and $X_2$;
similarly we use $\Gamma$ to refer to both graphs $\Gamma_1$ and
$\Gamma_2$.

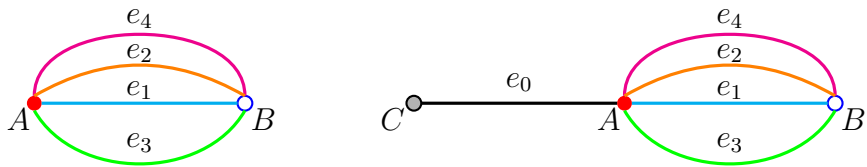
\begin{figure}[h]
  \begin{center}
    \begin{tikzpicture}[thick,scale=0.7]
\coordinate (A) at (4,0);
\coordinate (B) at (8,0);
\fill[red] (A) circle (4pt);
\draw[draw = blue, fill = blue!0,
                           inner sep = 0pt]  (B) circle (4pt);
%\node [circle,draw,inner sep=2pt,minimum size=6pt] (B) at (4,0) {};
\draw (4,0.14) [very thick, color=magenta] to[out=90,in=90] (8,0.14);
\node at (3.7,-0.3) {$A$};
\node at (8.35,-0.3) {$B$};
\node at (6,1.6)  {$e_4$};
\draw (4,0.14) [very thick, color=orange] to[out=30,in=150] (8,0.14);
\node at (6,0.95)  {$e_2$};
\draw (4.14,0) [very thick, color=cyan] --  (8-0.14,0);
\node at (6,0.25)  {$e_1$};
\draw (4,-0.14) [very thick, color=green] to[out=-60,in=-120] (8,-0.14);
\node at (6,-.75)  {$e_3$};
\end{tikzpicture}
    \hspace{20pt}
     \begin{tikzpicture}[thick,scale=0.7]
\coordinate (A) at (0,0);
\coordinate (B) at (4,0);
\coordinate (C) at (8,0);
\draw[draw = black, fill = black!30,
                           inner sep = 0pt] (A) circle (4pt);
\fill[red] (B) circle (4pt);
\draw[draw = blue, fill = blue!0,
                           inner sep = 0pt]  (C) circle (4pt);
%\node [circle,draw,inner sep=2pt,minimum size=6pt] (B) at (4,0) {};
\draw (0.14,0) [very thick]  -- (4-.14,0) node[midway,above] {$e_0$};
\draw (4,0.14) [very thick, color=magenta] to[out=90,in=90] (8,0.14);
\node at (3.7,-0.3) {$A$};
\node at (-0.4,-0.3) {$C$};
\node at (8.35,-0.3) {$B$};
\node at (6,1.6)  {$e_4$};
\draw (4,0.14) [very thick, color=orange] to[out=30,in=150] (8,0.14);
\node at (6,0.95)  {$e_2$};
\draw (4.14,0) [very thick, color=cyan] --  (8-0.14,0);
\node at (6,.25)  {$e_1$};
\draw (4,-0.14) [very thick, color=green] to[out=-60,in=-120] (8,-0.14);
\node at (6,-.75)  {$e_3$};
\end{tikzpicture}
    \caption{The graph
      $\Gamma_1$ (left) and the graph $\Gamma_2$ (right). In both graphs, the vertices $A$, $B$ correspond to red-filled and blue-empty type vertices in $X$, while the vertex $C$ corresponds to the decorated vertices in the black tails in $X_2$. Here $e_0$ is the line $CA$ and $e_j$,
      $1 \leq j \leq 4$ are the corresponding edges between the two
      vertices $A$ and $B$.}
    \label{fig:CompactGraph}
  \end{center}
\end{figure}

The graphs $X$ are metric graphs: each edge $e$ in $X$ is
identified with the interval $[0,\ell(e)]$, where $\ell(e)$ is the
length of the edge $e$.  The lengths of edges related by a periodic
shift (i.e.\ belonging to the same $\mathbb{Z}^3$-equivalence class or
having the same color) are the same.  We denote by $\ell_j$,
$j\in\{1,\ldots,4\}$ the distinct lengths of edges in the graph $X_1$;
the graph $X_2$ has an additional length --- the length of the tail
--- which we denote by $\ell_0$.  This metric information on $X$ can
be viewed as a pull-back of the metric on $\Gamma$ via the covering
map $\pi: X \rightarrow \Gamma$.  Notice that unlike the periodic
realization of graphene and its multi-layer variants, we do not assume
that each hexagon in the layer graph is regular, i.e.\ the lengths of
edges with distinct colors may be different.

On the edges of the graph $X$ we consider the Laplacian
$-\Delta_X=-\frac{d^2}{dx^2}$ or, more generally, the Schr\"odinger
operator $-\Delta_X + q_e(x)$ with piecewise continuous potential
$q_e(x)$.  The potential is assumed to be the same on the edges of the
same equivalence class (color), taking into account the edge's
orientation.  This ensures the potential is $\mathbb{Z}^3$-periodic
like the rest of the graph; we do not impose any other symmetry
conditions on $q_e$.  Regularity of the potential also plays no role
in our examples, the same results can be extended to $L^1$ potentials
with minor modifications.

At every vertex of the graph $X_2$, we impose the standard
Neumann-Kirchhoff boundary condition; we impose $\delta$-type
conditions with distinct coupling constants $\gamma_A$ and $\gamma_B$ (one of
them may be zero) on the corresponding vertices of the graph $X_1$.
For the precise definition of vertex conditions, the reader is
referred to Section~\ref{sec:qg_vc}.  The graphs $X$ are non-compact,
$\mathbb{Z}^3$-periodic quantum graphs.  According to the
Floquet-Bloch theory, the spectrum of the operator $-\Delta_X$ is the
union of the ranges of the band functions $\lambda_j = \lambda_j(k)$,
$j \geq 1$, where the \emph{quasimomentum} $\veck$ ranges over the
torus
$\mathbb{T}^3 :=
\left(\mathbb{R}/2\pi\mathbb{Z}\right)^3=(-\pi,\pi]^3$ and
\begin{equation}
  \label{eq:ev_numbering}
  \lambda_1(\veck) \leq \lambda_2(\veck) \leq \cdots
  \qquad\mbox{for any } \veck\in \mathbb{T}^3.
\end{equation}
Now we state our main result.

\begin{thm}
  \label{main}
  \begin{enumerate}
  \item[(a)] The spectrum of the operator $-\Delta_X$ has an open gap
    between the first and the second band functions, i.e.
    \begin{equation*}
      \max_\veck \lambda_1(\veck)<\min_\veck \lambda_2(\veck).
    \end{equation*}
  \item[(b)] If the lengths $\ell_j$ ($1 \leq j\leq 4$) are
    approximately equal, then there exists a non-trivial
    one-dimensional algebraic curve $\mu$ in $\mathbb{T}^3$ such that
    $\lambda_1$ attains its maximum value on $\mu$.  Consequently,
    there exists a \textbf{degenerate band edge} in the spectrum of
    $-\Delta_X$.
  \item[(c)] The degenerate band edge in the spectrum is persistent
    under a small perturbation of edge lengths, vertex
    coupling constants or edge potentials.
  \end{enumerate}
\end{thm}

Theorem~\ref{main} will be proved in Section~\ref{sec:main_proof}
after reviewing relevant definitions and tools in
Section~\ref{prelim}.  It will become clear during the proof that the
phenomenon described in the Theorem is very robust.  Informally
speaking, the extremum responsible for a band edge is frequently
degenerate for any graph where removing a single vertex (but not the edges
incident to it) reduces the rank of the fundamental group by 3 or
more.  In particular, the condition on the edge lengths in part (b) of
the Theorem serves only to insure a degenerate band edge particularly
for the first band.  For almost all choices of edge lengths one can
show that a \emph{finite proportion} of bands will have degenerate
edges.\footnote{This is a consequence of Barra--Gaspard ergodicity of
  quantum graphs: informally, what happens once for one choice of
  lengths will happen with finite frequency for almost all choices of
  lengths.  For more precise statements, see
  \cites{BarGas_jsp00,BerWin_tams10,BanBer_prl13,ExnTur_jpa17}}

The decorations introduced at vertices $A$ and $B$ ($\delta$-type
conditions in $X_1$ and the tail edge in $X_2$) serve to break
symmetry in the periodic graph and thus create a band gap.  If the
symmetry is not broken, one would expect the bands to touch along the
curve $\mu$; see \cite{BerCom_jst18} for related results.  Finally, the
topology of the degeneracy submanifold $\mu$ may be non-trivial in
the higher-dimensional analogues of our example.  We touch upon it in
in Section~\ref{sec:discussion}.

%%%%%%%%%%%%%%%%%%%%%%%%%%%% NEW SECTION %%%%%%%%%%%%%%%%%%%%%%%%%%%%%%
\section{Some preliminaries and notations}
\label{prelim}

%%%%%%%%%%%
\subsection{Quantum graphs and vertex conditions}
\label{sec:qg_vc}

In this section we recall some notations and basic notions of quantum
graphs; for more details the reader is encouraged to consult
\cite{BK,Mugnolo_book}.  Consider a graph $\mathcal{G}=(\mathcal{V},\mathcal{E})$
where $\mathcal{V}$ and $\mathcal{E}$ are the sets of vertices and
edges of $\mathcal{G}$, respectively.  For each vertex
$v\in\mathcal{V}$, let $\mathcal{E}_v$ be the set of edges $e$
incident to the vertex $v$.  The degree $d_v$ of the vertex $v$ is the
cardinality of the set $\mathcal{E}_v$.  The graph $\mathcal{G}$ is a
\emph{metric graph} if each edge $e$ of the graph is give a length,
$\ell_e$ and can thus be identified with the interval $[0,\ell_e]$.  A
function $f$ on the graph $\mathcal{G}$ is henceforth a collection of
functions $\{f_e\}_{e\in\mathcal{E}}$, each defined on the
corresponding interval.

Let us denote by $L^2(\mathcal{G})$ (correspondingly
$H^2(\mathcal{G})$) the space of functions on the graph $\mathcal{G}$
such that on each edge $e$ in $\mathcal{E}$, $f_e$ belongs to $L^2(e)$
(corresp.\ $H^2(e)$) and, moreover,
\begin{equation*}
  \sum_{e \in \mathcal{E}} \|f\|^2_{L^2(e)}<\infty
  \qquad \left(
    \mbox{corresp.\ }
    \sum_{e \in \mathcal{E}} \|f\|^2_{H^2(e)}<\infty
  \right).
\end{equation*}
$\mathcal{G}$ is called a \textit{quantum graph} if it is a metric
graph equipped with a self-adjoint differential operator $\mathcal{H}$
of the Schr\"odinger type acting in $L^2(\mathcal{G})$.  We will take
$\mathcal{H}$ to act as $-\Delta_{\mathcal{G}}+q_e(x)$ on the edge
$e$, where $q_e$ are assumed to be piecewise continuous.  The domain
of the operator will be the Sobolev space $H^2(\mathcal{G})$ further
restricted by a set of \emph{vertex conditions} which involve the
values of $f_e(v)$ and the derivatives $\frac{df_e}{dx}(v)$ calculated
at the vertices.  We list some commonly used vertex conditions below.

\begin{itemize}
\item \textit{Dirichlet} condition at a vertex $v \in \mathcal{V}$
  requires that the function $f$ vanishes at the vertex,
  \begin{equation*}
    f(v) = 0.
  \end{equation*}
  This is an example of a decoupling condition. Namely, if the Dirichlet
  condition is imposed at a vertex of degree $d>1$, it is equivalent to
  disconnecting the edges incident to the vertex and imposing Dirichlet
  conditions at the resulting $d$ vertices of degree 1.
\item \emph{$\delta$-type} condition at a vertex $v\in\mathcal{V}$
  requires the function to be continuous at $v$ in addition to the condition
  \begin{equation}
    \label{eq:delta-type}
    \sum_{e \in \mathcal{E}_v}\frac{df_e}{dx}(v) = \gamma_v f(v),
    \qquad
    \gamma_v\in\mathbb{R},
  \end{equation}
  where $\frac{df_e}{dx}(v)$ is the derivative of the function $f_e$
  taken in the direction into the edge.  We note that the value $f(v)$
  is well-defined because of the assumed continuity.  The real
  parameter $\gamma_v$ is called the \emph{vertex coupling constant}.
  The special case of the $\delta$-type condition with $\gamma_v=0$ is
  the \textit{Neumann-Kirchhoff} (NK) or ``standard'' condition.  The
  Dirichlet condition defined above can be naturally interpreted as
  $\gamma_v=+\infty$.
\item \emph{quasi-NK} or \emph{magnetic} condition at a vertex
  $v\in\mathcal{V}$: Assume that the degree of the vertex $v$ is
  $d_v$, $\mathcal{E}_v=\{1, \ldots, d_v\}$ and we are given
  $d_v$ \textit{unit complex scalar}s
  $z_1, \ldots, z_{d_v} \in \mathbb{S}^1$. We impose the following two
  conditions:
  \begin{equation}
    \label{quasi-NK}
    \begin{cases}
      &z_1 f_1(v) = z_2 f_2(v) = \ldots = z_{d_v} f_{d_v}(v) \\
      &\sum_{j=1}^{d_v}z_j \frac{df_j}{dx}(v) = 0,
    \end{cases}
  \end{equation}
  Of course, the NK condition is a special case of~\eqref{quasi-NK}
  when all $z_j$ are equal.
\end{itemize}
If every vertex of the graph $\mathcal{G}$ is equipped with one of the
above conditions, the operator $\mathcal{H}$ is self-adjoint (see
\cite{BK}*{Theorem 1.4.4} and references therein).  The last set of
conditions allow one to introduce magnetic field on the graph without
modifying the operator (see \cite{KosSch_cmp03} and also
\cite{Mugnolo_book,Kur_lmp19} for more recent appearances).  They also
arise as a result of Floquet--Bloch reduction reviewed in the next
section.

\subsection{Floquet-Bloch reduction}
\label{sec:fb}

Let us now return to our periodic graph $X$. Recall that the
$\delta$-type conditions are imposed at all vertices of $X$ and hence
the operator $-\Delta_X$ is self-adjoint. A standard Floquet-Bloch
reduction (see e.g., \cites{BK, Kbook, Ksurvey}) allows us to reduce
the consideration of the spectrum of $-\Delta_X$ to a family of
spectral problems on a compact quantum graph (a fundamental
domain). More precisely, denote by $g_1, g_2, g_3$ some choice of
generators of the shift lattice $\mathbb{Z}^3$.  For each
$\veck=(k_1, k_2, k_3) \in (-\pi,\pi]^3 =: \mathbb{T}^3$, let
$-\Delta_X^{(\veck)}$ be the Laplacian that acts on the domain consisting
of functions $u \in H^2_{loc}(X)$ that satisfy the $\delta$-type
conditions at vertices along with the following Floquet conditions,
\begin{equation}
  \label{eq:cond_floquet}
  u_{g_1e}(x) = e^{ik_1}u_e(x),
  \quad
  u_{g_2e}(x) = e^{ik_2}u_e(x),
  \quad
  u_{g_3e}(x) = e^{ik_3}u_e(x),
\end{equation}
for all $x \in X$ and $n=(n_1, n_2, n_3) \in \mathbb{Z}^3$. Then
$-\Delta_X$ is the direct integral of $-\Delta_X^{(\veck)}$ and therefore,
\begin{equation}
  \label{eq:spec_repres}
  \sigma(-\Delta_X) = \bigcup_{\veck \in \mathbb{T}^3}\sigma(-\Delta_X^{(\veck)}).  
\end{equation}
The operator $-\Delta_X^{(\veck)}$ has discrete spectrum
$\sigma(-\Delta_X^{(\veck)})=\{\lambda_j(\veck)\}_{j=1}^\infty$ where we
assume that $\lambda_j$ is increasing in $j$, see
\eqref{eq:ev_numbering}.  The \emph{dispersion relation} of the
operator $-\Delta_X$ is the multivalued function
$\veck \mapsto \{\lambda_j(\veck)\}$ and the spectrum of $-\Delta_X$ is the
range of the dispersion relation for quasimomentum $\veck$ in
$\mathbb{T}^3$. Hence, it suffices to focus on solving the eigenvalue
problems $-\Delta_X^{(\veck)} u =\lambda u$ where $\lambda \in \mathbb{R}$
for $u$ in the domain of $-\Delta_X^{(\veck)}$.   This problem
is unitarily equivalent to the eigenvalue problem on the compact graph
$\Gamma$,
\begin{equation}
\label{eigenvalue-prob}
-\frac{d^2}{dx^2}u=\lambda u, \quad \lambda \in \mathbb{R},
\end{equation}
where $u$ satisfies the respective vertex conditions at the vertices
$A$ and $C$ and the quasi-NK conditions at the vertex $B$:
\begin{equation}
  \label{quasi-conditions}
  \begin{cases}
    &e^{ik_1}u_1(B)=e^{ik_2}u_2(B)=e^{ik_3}u_3(B)=u_4(B) \\
    &e^{ik_1}u_1'(B)+e^{ik_2}u_2'(B)+e^{ik_3}u_3'(B)+u_4'(B)=\gamma_B u(B),
  \end{cases}
\end{equation}
where $\gamma_B$ is taken to be 0 for the graph $\Gamma_2$ and $u_j$
are the restrictions of the function $u$ to the edges $e_j$.  We will
use the notation $\Gk_1$ and $\Gk_2$ (or $\Gk$ if the distinction
between the two graphs is irrelevant) to denote the eigenvalue problem
with condition~\eqref{quasi-conditions} at the vertex $B$.

From now on, we shall emphasize the vertex conditions pictorially by
replacing the names of the vertices by their corresponding boundary
conditions, see Fig.~\ref{fig:QuasiGamma}. 
We will use $\gamma_A$,
$\mathbf{NK}$, $\mathbf{D}$ and $\mathbf{Q_{k,\gamma_B}}$ to indicate the
$\delta$-type, Neumann--Kirchhoff, Dirichlet and quasi-NK vertex
conditions respectively.  We will also occasionally use this
convention in the text, e.g., the vertex $B$ in the above graph $\Gk$
will be mentioned as the $\mathbf{Q_k}$-vertex.  Finally, we will use
the symbol $\lambda_j(\Gk)$ for the $j^{th}$-eigenvalue of the quantum
graph $\Gk$. In particular, we have
\begin{equation}
  \label{eq:fb_decomposition}
  \sigma(-\Delta_X) = \bigcup_{j \geq 1, k \in
  \mathbb{T}^3}\left\{\lambda_j(\Gk)\right\}.  
\end{equation}

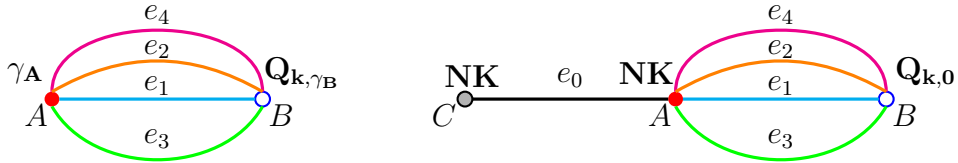
\begin{figure}[h]
  \begin{center}
    \begin{tikzpicture}[thick,scale=0.7]
%   \node[circle,draw] (B) at (5,0) {};
%  \draw  (-5,0) node {}  -- (-1,0) node{} -- (5,0) node{};
%  \draw (A) to[out=20,in=160] (B);
%  \draw (A) to[out=40,in=120] (B);
%  \draw (A) to[out=-20,in=200] (B);
%  \draw (A) to[out=-40,in=220] (B);
\coordinate (A) at (4,0);
\coordinate (B) at (8,0);
\fill[red] (A) circle (4pt);
\draw[draw = blue, fill = blue!0,
                           inner sep = 0pt]  (B) circle (4pt);
%\node [circle,draw,inner sep=2pt,minimum size=6pt] (B) at (4,0) {};
\draw (4,0.14) [very thick, color=magenta] to[out=90,in=90] (8,0.14);
\node at (3.5,0.5) {$\mathbf{\gamma_A}$};
\node at (3.7,-0.3) {$A$};
\node at (8.35,-0.3) {$B$};
\node at (6,1.6)  {$e_4$};
\draw (4,0.14) [very thick, color=orange] to[out=30,in=150] (8,0.14);
\node at (6,0.95)  {$e_2$};
\draw (4.14,0) [very thick, color=cyan] --  (8-0.14,0);
\node at (8.75,0.5)  {$\mathbf{Q_{k,\gamma_B}}$};
\node at (6,0.25)  {$e_1$};
\draw (4,-0.14) [very thick, color=green] to[out=-60,in=-120] (8,-0.14);
\node at (6,-.75)  {$e_3$};

%\fill[black] (0,0) circle (1pt) (4,0) circle (1pt) (8,0) circle (1pt);
%\coordinate (A) at (0,0);
%\coordinate (B) at (4,0);
%\coordinate (C) at (8,0);
%\fill[black] (A) circle (2pt) (B) circle (2pt) (C) circle (2pt);
%\draw (A) node[above]{$\textbf{NK}$}  -- (B) node[midway,above] {$e_0$};
%\draw (B) node[above, xshift = -0.4cm]{$\textbf{NK}$} to[out=90,in=90] node[midway,above] {$e_4$} (C);
%\draw (B) to[out=30,in=150] node[midway,above] {$e_2$} (C);
%\draw (B) -- node[midway,below] {$e_1$} (C) node[above, xshift = 0.3cm]{$\mathbf{Q_k}$};
%\draw (B) to[out=-60,in=-120] node[midway,below] {$e_3$} (C);
\end{tikzpicture}
    \hspace{15pt}
    \begin{tikzpicture}[thick,scale=0.7]
%   \node[circle,draw] (B) at (5,0) {};
%  \draw  (-5,0) node {}  -- (-1,0) node{} -- (5,0) node{};
%  \draw (A) to[out=20,in=160] (B);
%  \draw (A) to[out=40,in=120] (B);
%  \draw (A) to[out=-20,in=200] (B);
%  \draw (A) to[out=-40,in=220] (B);
\coordinate (A) at (0,0);
\coordinate (B) at (4,0);
\coordinate (C) at (8,0);
\draw[draw = black, fill = black!30,
                           inner sep = 0pt] (A) circle (4pt);
\fill[red] (B) circle (4pt);
\draw[draw = blue, fill = blue!0,
                           inner sep = 0pt]  (C) circle (4pt);
%\node [circle,draw,inner sep=2pt,minimum size=6pt] (B) at (4,0) {};
\draw (0.14,0) [very thick] node[above]{$\textbf{NK}$} -- (4-.14,0) node[midway,above] {$e_0$};
\draw (4,0.14) [very thick, color=magenta] to[out=90,in=90] (8,0.14);
\node at (3.45,0.5) {$\textbf{NK}$};
\node at (3.7,-0.3) {$A$};
\node at (-0.4,-0.3) {$C$};
\node at (8.35,-0.3) {$B$};
\node at (6,1.6)  {$e_4$};
\draw (4,0.14) [very thick, color=orange] to[out=30,in=150] (8,0.14);
\node at (6,0.95)  {$e_2$};
\draw (4.14,0) [very thick, color=cyan] --  (8-0.14,0);
\node at (8.75,0.5)  {$\mathbf{Q_{k, 0}}$};
\node at (6,.25)  {$e_1$};
\draw (4,-0.14) [very thick, color=green] to[out=-60,in=-120] (8,-0.14);
\node at (6,-.75)  {$e_3$};

%\fill[black] (0,0) circle (1pt) (4,0) circle (1pt) (8,0) circle (1pt);
%\coordinate (A) at (0,0);
%\coordinate (B) at (4,0);
%\coordinate (C) at (8,0);
%\fill[black] (A) circle (2pt) (B) circle (2pt) (C) circle (2pt);
%\draw (A) node[above]{$\textbf{NK}$}  -- (B) node[midway,above] {$e_0$};
%\draw (B) node[above, xshift = -0.4cm]{$\textbf{NK}$} to[out=90,in=90] node[midway,above] {$e_4$} (C);
%\draw (B) to[out=30,in=150] node[midway,above] {$e_2$} (C);
%\draw (B) -- node[midway,below] {$e_1$} (C) node[above, xshift = 0.3cm]{$\mathbf{Q_k}$};
%\draw (B) to[out=-60,in=-120] node[midway,below] {$e_3$} (C);
\end{tikzpicture}
  \end{center}
  \caption{The quantum graphs $\Gk_1$ (left) and $\Gk_2$ (right) and their vertex
    conditions. In the figures, the types of the boundary conditions are bold letters while the labels of the vertices are regular letters.}
  \label{fig:QuasiGamma}
\end{figure}

%%%%%%%%%%%%
\subsection{Eigenvalue comparison under some surgery transformations}
\label{sec:surgery}

In this section we list some eigenvalue comparison results that will
be useful to prove the existence of a gap in the dispersion relation
in Theorem~\ref{main}(a).

The following interlacing inequality is often useful when variation of
a coupling constant is used to interpolate between different
$\delta$-type conditions and also the Dirichlet condition (which is
interpreted as the $\delta$-type condition with coupling $+\infty$).

\begin{thm}[A special case of {\cite{BKKM}*{Theorem 3.4}}]
  \label{rank1}
  If the graph $\widehat{G}$ is obtained from $G$ by changing the
  coefficient of the 
  $\delta$-type condition at a single vertex $v$ from $\gamma_v$ to
  $\widehat{\gamma}_v \in (\gamma_v, \infty]$.
  Then their eigenvalues satisfy the interlacing inequalities
  \begin{equation}
    \label{eq:interlacing}
    \lambda_{k}(G) \leq \lambda_{k}(\widehat{G}) 
    \leq \lambda_{k+1}(G) \leq \lambda_{k+1}(\widehat{G}),
    \qquad k \geq 1.
  \end{equation}
  If a given value $\Lambda$ has multiplicities $m$ and
  $\widetilde{m}$ in the spectra of $G$ and $\widehat{G}$
  respectively, then the $\Lambda$-eigenspaces of $G$ and
  $\widehat{G}$ intersect along a subspace of dimension
  $\min(m, \widetilde{m})$.  Note that by \eqref{eq:interlacing},
  $\widetilde{m}$ must be equal to $m-1$, $m$ or $m+1$.
\end{thm}

For simplicity, from now on, if the graph $G_1$ is obtained from
$G_2$ by changing the $\delta$-type conditions to Dirichlet conditions
at a single vertex, we will say that \textit{$G_1$ is a rank one
  Dirichlet perturbation of the graph $G_2$}.

We now consider the effect on the eigenvalue of the enlargement of a
graph, which is realized by attaching a subgraph at a designated
vertex.  The following theorem is quoted in the narrowest form that is
sufficient for our needs.

\begin{thm}[A special case of {\cite{BKKM}*{Theorem 3.10}}]
  \label{attach}
  Suppose that $\widehat{G}$ is formed from graphs $G$ and $H$ by
  identifying or ``gluing'' two Neumann--Kirchhoff vertices $v_0\in G$
  and $w_0\in H$.  If $\lambda_1(H)<\lambda_1(G)$ and the eigenvalue
  $\lambda_1(G)$ has an eigenfunction which does not vanish at $v_0$
  then $\lambda_1(\widehat{G})<\lambda_1(G)$.
\end{thm}

%%%%%%%%%%%
\subsection{Topology of moduli spaces of polygons}

Given $n$ positive real numbers $\{a_j\}$ one can ask what is the
topology of the space of all planar polygons whose side lengths are
$\{a_j\}$.  Two polygons are identified if they can be mapped into
each other by a composition of rotation and translation.  The
resulting spaces may not be smooth and their full classification is
surprisingly rich, see \cite{FarSch} and references
therein.  These spaces make an appearance in our question as the
degenerate curves on which the dispersion relation has an extremum.

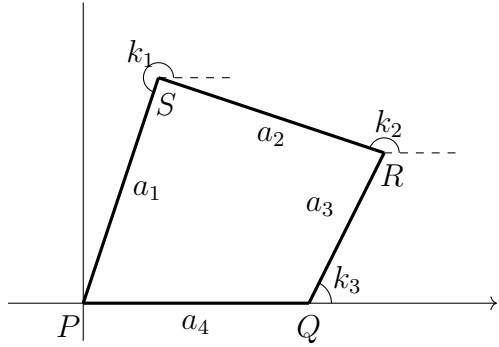
\begin{figure}
  \centering
  \begin{tikzpicture}
\draw[->] (-1,0)--(5.5,0);
\draw (0,-.5)--(0,4);
\coordinate (o) at (0,0);
\coordinate (a) at (3,0);
\coordinate (c) at (1,3);
\coordinate (b) at (4,2);
\coordinate (bx) at (5,2);
\coordinate (cx) at (2,3);
\coordinate (ax) at (5,0);
\draw[very thick] (o) node[left,below,xshift=-0.2cm]{$P$}--(a) node[midway,below] {$a_4$};
\draw[very thick] (a) node[right,below]{$Q$}--(b) node[midway,above left] {$a_3$};
\draw[very thick] (b) node[right,below, xshift=0.1cm]{$R$} --(c) node[midway,below] {$a_2$};
\draw[very thick] (c) node[below, xshift=0.1cm,yshift=-0.05cm]{$S$}--(o) node[midway,right] {$a_1$};
\draw[dashed] (b)--(bx) (c)--(cx);
\pic[draw,"$k_3$",angle radius=0.3cm,angle eccentricity=2] {angle=ax--a--b};
\pic[draw,"$k_1$",angle radius=0.2cm,angle eccentricity=2] {angle=cx--c--o};
\pic[draw,"$k_2$",angle radius=0.2cm,angle eccentricity=2] {angle=bx--b--c};
\end{tikzpicture}
  \caption{Quadrangle corresponding to
    equation~\eqref{eq:quadrangle_condition}.}
  \label{fig:quadrangle}
\end{figure}

For our example we will only require the following simple lemma (which
follows from the results of \cite{FarSch}) addressing the topology of
the set of quadrangles with given four edge lengths, see
Figure~\ref{fig:quadrangle}.

\begin{lemma}
  \label{lemma-quadrangle}
  The curve $\mu$ of solutions $\veck=(k_1,k_2,k_3)\in\mathbb{T}^3$ of
  \begin{equation}
    \label{eq:quadrangle_condition}
    \sum_{1 \leq j \leq 3} e^{ik_j}a_j+a_4=0  
  \end{equation}
  is an algebraic curve of co-dimension 2 if and only if
  \begin{equation}
    \label{quadrangle-strict}
    a_m < \sum_{j\neq m} a_j
  \end{equation}
  for every $m=1,\ldots,4$.

  If there is an $m$ with the inequality reversed, the set of
  solutions $\mu$ is empty.  If there is an $m$ with inequality
  turning into equality, the set of solutions is a single point.
\end{lemma}

The topology of $\mu$ in this particular case has been described, for
example, in \cite{KapMil_jdg95}*{Sec 12}.  The curve is smooth unless
there is a linear combination of $\{a_j\}$ with coefficients $\pm1$
that is equal to zero.  If the curve $\mu$ is smooth it is either a
circle or a disjoint union of two circles.  The non-generic cases when
$\mu$ is not a smooth manifold are of the following types: two circles
intersecting at a point, two circles intersecting at two points and
three circles with one intersection among each pair.

\begin{figure}
  \centering
  \includegraphics[scale=0.4]{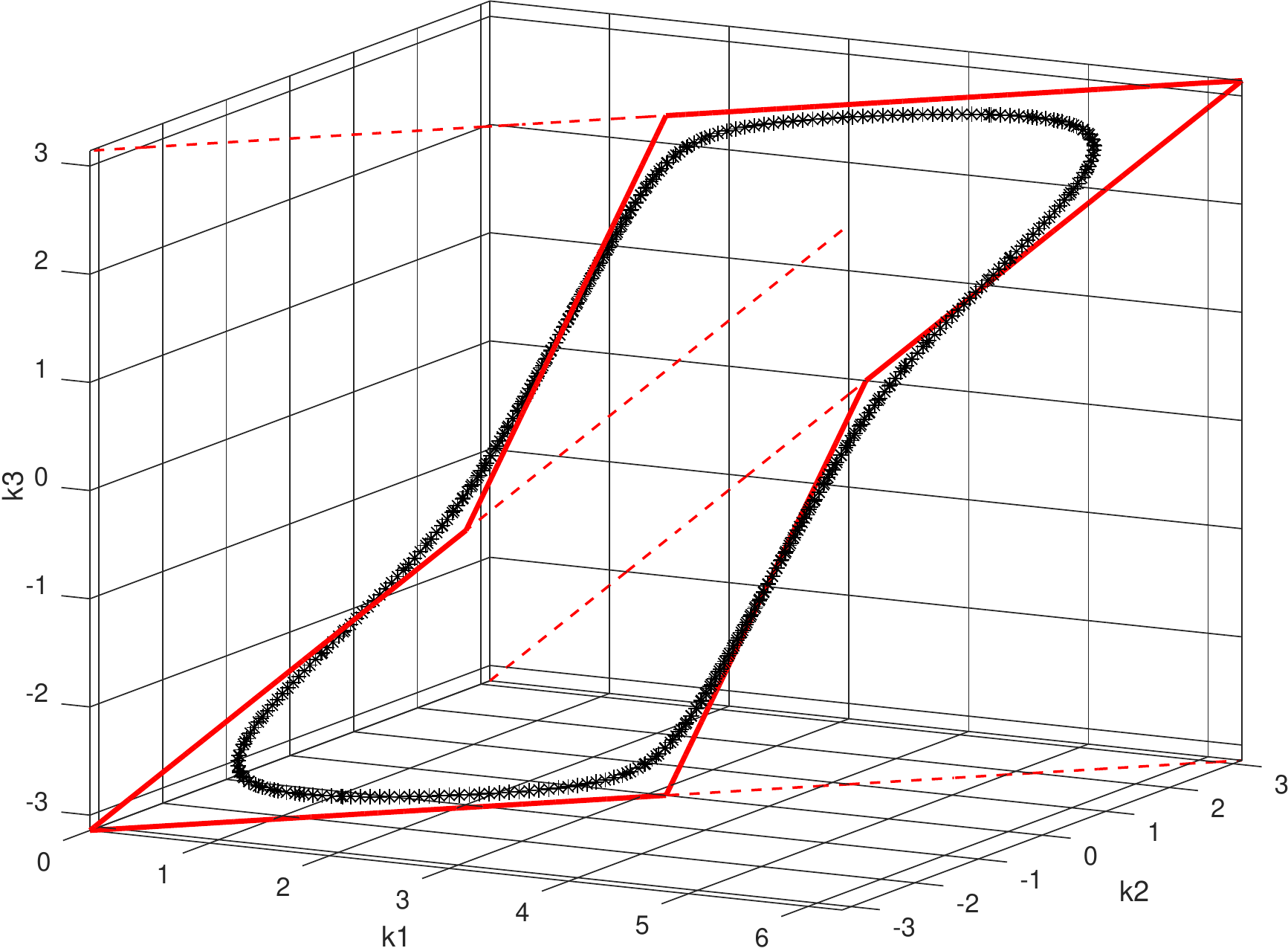}
  \includegraphics[scale=0.4]{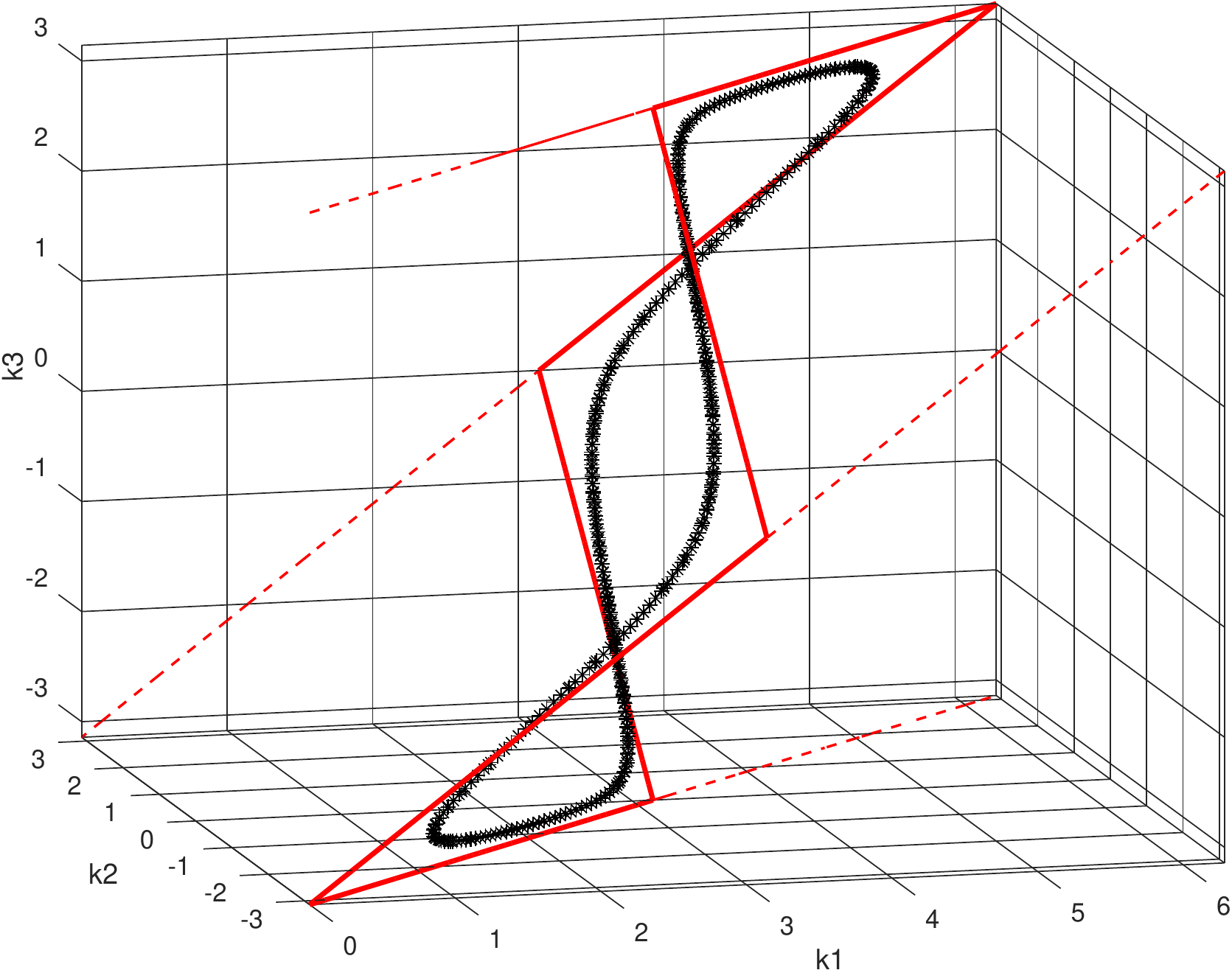}  
  \caption{The set of roots of \eqref{eq:quadrangle_condition} for two
    choices of $\{a_j\}$; two views of the same plot are shown.  The
    ranges are adjusted to $k_1 \in (0,2\pi]$ and
    $k_2,k_3 \in (-\pi,\pi]$ for a smoother plot.  Straight red lines
    correspond to $a_j=1$ for all $j$; Black stars (appear as a thick
    fuzzy line) are produced using $a_1=1.1$, $a_2=0.95$, $a_3=0.9$
    and $a_4=1$.}
  \label{fig:QuadrangleSpace}
\end{figure}

The latter case arises when all $a_j$ are equal.  It is shown in red
solid line in Figure~\ref{fig:QuadrangleSpace}.  Note that the plot is
on a torus, therefore each pair of parallel lines is actually a single
line forming a circle.  A smooth curve $\mu$ for a
generic choice of $a_j\approx 1$ is also shown.

%%%%%%%%%%%%%%%%%%%%%%%%%%%%%%%%%% NEW SECTION %%%%%%%%%%%%%%%%%%%%%%%%%%%%%%
\section{Proof of the main result}
\label{sec:main_proof}

In this section, we present the details of the proof of Theorem
\ref{main}.  Without loss of generality, for the graph $\Gamma_1$ we
will make the assumption
\begin{equation}
  \label{eq:gammas_relation}
  \gamma_A < \gamma_B.
\end{equation}

Starting with the graph $\Gamma_1$, we introduce two of its
modifications.  The graph $\Gamma_1^A$ is obtained by changing the
condition at the vertex $A$ to Dirichlet; the graph $\Gamma_1^B$ is
obtained similarly by placing a Dirichlet condition at the vertex
$B$.  Remembering that a Dirichlet condition is decoupling, we can
picture the result as shown in Fig.~\ref{fig:GammaAB1}.

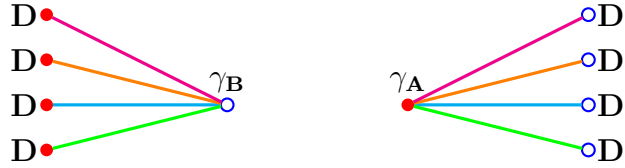
\begin{figure}[h]
  \centering
  \begin{tikzpicture}[thick,scale=0.6]
%   \node[circle,draw] (B) at (5,0) {};
%  \draw  (-5,0) node {}  -- (-1,0) node{} -- (5,0) node{};
%  \draw (A) to[out=20,in=160] (B);
%  \draw (A) to[out=40,in=120] (B);
%  \draw (A) to[out=-20,in=200] (B);
%  \draw (A) to[out=-40,in=220] (B);
\coordinate (A1) at (0,-1);
\coordinate (A2) at (0,0);
\coordinate (A3) at (0,1);
\coordinate (A4) at (0,2);
\coordinate (B) at (4,0);
\fill[red] (A1) circle (4pt);
\fill[red] (A2) circle (4pt);
\fill[red] (A3) circle (4pt);
\fill[red] (A4) circle (4pt);
\draw[draw = blue, fill = blue!0,
                           inner sep = 0pt]  (B) circle (4pt);
\draw (4-.14,0) [very thick, color = cyan] -- (0.14,0);
\draw (4-.135,0.034) [very thick, color = orange] -- (.135,1-.034);
\draw (4-.125,.062) [very thick, color = magenta] -- (.125, 2-.062);
\draw (4-.135,-.034) [very thick, color = green] -- (.135,-1+.034);

\node at (4,.5) {$\mathbf{\gamma_B}$};
\node at (-.5,-1) {$\textbf{D}$};
\node at (-.5,0) {$\textbf{D}$};
\node at (-.5,1) {$\textbf{D}$};
\node at (-.5,2) {$\textbf{D}$};
%\node [circle,draw,inner sep=2pt,minimum size=6pt] (B) at (4,0) {};
%\draw (4,0.14) [very thick, color=magenta] to[out=90,in=90] (8,0.14);
%\node at (3.5,0.5) {$\mathbf{\gamma_A}$};
%\node at (3.7,-0.3) {$A$};
%\node at (8.35,-0.3) {$B$};
%\node at (6,1.6)  {$e_4$};
%\draw (4,0.14) [very thick, color=orange] to[out=30,in=150] (8,0.14);
%\node at (6,0.95)  {$e_2$};
%\draw (4.14,0) [very thick, color=cyan] --  (8-0.14,0);
%\node at (8.5,0.5)  {$\mathbf{Q_k}$};
%\node at (6,0.25)  {$e_1$};
%\draw (4,-0.14) [very thick, color=green] to[out=-60,in=-120] (8,-0.14);
%\node at (6,-.75)  {$e_3$};
\end{tikzpicture}
  \hspace{35pt}
  \begin{tikzpicture}[thick,scale=0.6]
\coordinate (B1) at (4,-1);
\coordinate (B2) at (4,0);
\coordinate (B3) at (4,1);
\coordinate (B4) at (4,2);
\coordinate (A) at (0,0);
\draw[draw = blue, fill = blue!0,
                           inner sep = 0pt] (B1) circle (4pt);
\draw[draw = blue, fill = blue!0,
                           inner sep = 0pt] (B2) circle (4pt);
\draw[draw = blue, fill = blue!0,
                           inner sep = 0pt](B3) circle (4pt);
\draw[draw = blue, fill = blue!0,
                           inner sep = 0pt] (B4) circle (4pt);
\fill[red]  (A) circle (4pt);
\draw (4-.14,0) [very thick, color = cyan] -- (0.14,0);
\draw (.136,.034) [very thick, color = orange] -- (4-.136,1-.034);
\draw (4-.125,2-.062) [very thick, color = magenta] -- (.125, .062);
\draw (4-.135,-1+.034) [very thick, color = green] -- (.135,-.034);

\node at (0,.5) {$\mathbf{\gamma_A}$};
\node at (4+.5,-1) {$\textbf{D}$};
\node at (4+.5,0) {$\textbf{D}$};
\node at (4+.5,1) {$\textbf{D}$};
\node at (4+.5,2) {$\textbf{D}$};
\end{tikzpicture}
  \caption{The two ``star'' graphs $\Gamma^A_1$ (left)
    and $\Gamma^B_1$ (right) after disconnecting the corresponding Dirichlet vertices $(\textbf{D})$.}
  \label{fig:GammaAB1}
\end{figure}

By placing Dirichlet conditions at vertices $A$ or $B$ of the graph
$\Gamma_2$, we analogously construct the two graphs $\Gamma_2^A$ and
$\Gamma_2^B$.  We remark that the graph $\Gamma_2^A$ has two connected
components, see Fig.~\ref{fig:GammaA2} and Fig.~\ref{fig:GammaB2}.  Using the tools introduced in Section~\ref{sec:surgery}
we establish the following comparison result, which compares the first
eigenvalue $\Gamma_j^A$ with the first eigenvalue of $\Gamma_j^B$,
where $j$ is either 1 or 2.

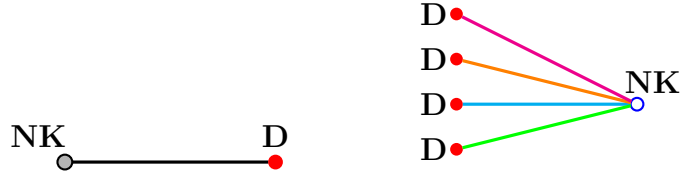
\begin{figure}[h]
  \begin{center}
    \begin{tikzpicture}[thick,scale=0.7]
%   \node[circle,draw] (B) at (5,0) {};
%  \draw  (-5,0) node {}  -- (-1,0) node{} -- (5,0) node{};
%  \draw (A) to[out=20,in=160] (B);
%  \draw (A) to[out=40,in=120] (B);
%  \draw (A) to[out=-20,in=200] (B);
%  \draw (A) to[out=-40,in=220] (B);
\coordinate (A2) at (0,0);
\coordinate (B) at (4,0);
\fill[red] (B) circle (4pt);
\draw[draw = black, fill = black!30,
                           inner sep = 0pt]  (A2) circle (4pt);
\draw (4-.14,0) [very thick, color = black] -- (0.14,0);

\node at (4,.5) {$\textbf{D}$};
\node at (-.5,.5) {$\textbf{NK}$};
\end{tikzpicture}
    \hspace{30pt}
    \begin{tikzpicture}[thick,scale=0.6]
%   \node[circle,draw] (B) at (5,0) {};
%  \draw  (-5,0) node {}  -- (-1,0) node{} -- (5,0) node{};
%  \draw (A) to[out=20,in=160] (B);
%  \draw (A) to[out=40,in=120] (B);
%  \draw (A) to[out=-20,in=200] (B);
%  \draw (A) to[out=-40,in=220] (B);
\coordinate (A1) at (0,-1);
\coordinate (A2) at (0,0);
\coordinate (A3) at (0,1);
\coordinate (A4) at (0,2);
\coordinate (B) at (4,0);
\fill[red] (A1) circle (4pt);
\fill[red] (A2) circle (4pt);
\fill[red] (A3) circle (4pt);
\fill[red] (A4) circle (4pt);
\draw[draw = blue, fill = blue!0,
                           inner sep = 0pt]  (B) circle (4pt);
\draw (4-.14,0) [very thick, color = cyan] -- (0.14,0);
\draw (4-.135,0.034) [very thick, color = orange] -- (.135,1-.034);
\draw (4-.125,.062) [very thick, color = magenta] -- (.125, 2-.062);
\draw (4-.135,-.034) [very thick, color = green] -- (.135,-1+.034);

\node at (4.35,.5) {$\textbf{NK}$};
\node at (-.5,-1) {$\textbf{D}$};
\node at (-.5,0) {$\textbf{D}$};
\node at (-.5,1) {$\textbf{D}$};
\node at (-.5,2) {$\textbf{D}$};
%\node [circle,draw,inner sep=2pt,minimum size=6pt] (B) at (4,0) {};
%\draw (4,0.14) [very thick, color=magenta] to[out=90,in=90] (8,0.14);
%\node at (3.5,0.5) {$\mathbf{\gamma_A}$};
%\node at (3.7,-0.3) {$A$};
%\node at (8.35,-0.3) {$B$};
%\node at (6,1.6)  {$e_4$};
%\draw (4,0.14) [very thick, color=orange] to[out=30,in=150] (8,0.14);
%\node at (6,0.95)  {$e_2$};
%\draw (4.14,0) [very thick, color=cyan] --  (8-0.14,0);
%\node at (8.5,0.5)  {$\mathbf{Q_k}$};
%\node at (6,0.25)  {$e_1$};
%\draw (4,-0.14) [very thick, color=green] to[out=-60,in=-120] (8,-0.14);
%\node at (6,-.75)  {$e_3$};
\end{tikzpicture}
  \end{center}
  \caption{The quantum graphs $\Gamma_2^{A,1}$ (left) and $\Gamma_2^{A,2}$ 
  after disconnecting from the Dirichlet vertex of their union $\Gamma_2^A$.}
  \label{fig:GammaA2}
\end{figure}

\begin{figure}[h]
  \begin{center}
    \begin{tikzpicture}[thick,scale=0.6]
\coordinate (B1) at (4,-1);
\coordinate (B2) at (4,0);
\coordinate (B3) at (4,1);
\coordinate (B4) at (4,2);
\coordinate (A) at (0,0);
\coordinate (C) at (-4,0);
\draw[draw = blue, fill = blue!0,
                           inner sep = 0pt] (B1) circle (4pt);
\draw[draw = blue, fill = blue!0,
                           inner sep = 0pt] (B2) circle (4pt);
\draw[draw = blue, fill = blue!0,
                           inner sep = 0pt](B3) circle (4pt);
\draw[draw = blue, fill = blue!0,
                           inner sep = 0pt] (B4) circle (4pt);
\fill[red]  (A) circle (4pt);
\draw[draw = black, fill = black!30,
                           inner sep = 0pt] (C) circle (4pt);

\draw (-4+.14,0) [very thick, color = black] -- (-.14,0);
\draw (4-.14,0) [very thick, color = cyan] -- (0.14,0);
\draw (.136,.034) [very thick, color = orange] -- (4-.136,1-.034);
\draw (4-.125,2-.062) [very thick, color = magenta] -- (.125, .062);
\draw (4-.135,-1+.034) [very thick, color = green] -- (.135,-.034);

\node at (-4,.65) {\textbf{NK}};
\node at (-.25,.65) {$\textbf{NK}$};
\node at (4+.5,-1) {$\textbf{D}$};
\node at (4+.5,0) {$\textbf{D}$};
\node at (4+.5,1) {$\textbf{D}$};
\node at (4+.5,2) {$\textbf{D}$};
\end{tikzpicture}
  \end{center}
  \caption{The quantum graph $\Gamma_2^B$.}
  \label{fig:GammaB2}
\end{figure}
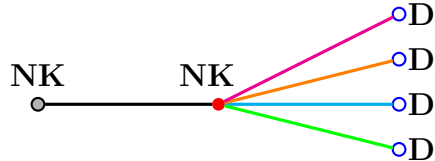

\begin{lemma}
  \label{GammaA-B}
  The first eigenvalue of $\Gamma^B$ is always strictly less than than
  the first eigenvalue of $\Gamma^A$,
  \begin{equation}
    \label{eq:gap_condition}
    \lambda_1(\Gamma^B)<\lambda_1(\Gamma^A)
  \end{equation}
\end{lemma}

\begin{proof}
  The graphs $\Gamma_1^B$ and $\Gamma_1^A$ differ only in the
  coefficient of the $\delta$-type condition at the vertex of degree 4
  (we are in the situation of pure Laplacian, with no potential).
  Since the coefficient of $\Gamma_1^B$ (which is $\gamma_A$) is
  smaller than the coefficient of $\Gamma_1^A$, see
  equation~\eqref{eq:gammas_relation}, we immediately get from
  Theorem~\ref{rank1} that $\lambda_1(\Gamma_1^B) \leq
  \lambda_1(\Gamma_1^A)$.  The case of equality is excluded because
  the ground state must be non-zero on the vertex of degree 4 which
  means it cannot satisfy $\delta$-type conditions with two different
  constants (hence it cannot be a common eigenfunction).

  For the graph $\Gamma_2$ we establish two inequalities,
  $\lambda_1(\Gamma_2^B) < \lambda_1(\Gamma_2^{A,1})$ and
  $\lambda_1(\Gamma_2^B) < \lambda_1(\Gamma_2^{A,2})$.  The first
  follows by changing the condition at vertex $A$ of the graph
  $\Gamma_2^B$ from NK to Dirichlet: the eigenvalue strictly increases
  (since the eigenfunction of $\Gamma_2^B$ is non-zero at $A$) and the
  graph decouples into several disjoint parts one of which coincides
  with $\Gamma_2^{A,1}$.

  To prove the second inequality, we start with
  $\lambda_1(\Gamma_2^{A,2})>0$ 
  whose eigenfunction does not vanish on vertex $B$, and attach to $B$
  a Neumann interval of length $\ell_0$ whose first eigenvalue is $0
  <\lambda_1(\Gamma_2^{A,2})$.  The strict inequality follows from
  Theorem~\ref{attach}.
\end{proof}

In our terminology, the graphs $\Gamma^A$ and $\Gamma^B$ are the rank
one Dirichlet perturbations of the corresponding graph $\Gamma$.  The
next important observation is that they are also, in fact, the rank
one Dirichlet perturbations of the corresponding graph $\Gamma^\veck$
for any $\veck$.

\begin{lemma}
  \label{lem:unitary_equivalence}
  The rank one Dirichlet perturbation of the graph $\Gamma^\veck$ at
  the vertex $A$ (corresp.\ $B$) is unitarily equivalent to $\Gamma^A$
  (corresp.\ $\Gamma^B$) for any $\veck \in \mathbb{T}^3$.
\end{lemma}

\begin{proof}
  Since the Dirichlet perturbation is decoupling, the resulting graphs
  have no cycles and therefore any quasi-momenta can be removed by a
  gauge transform, see \cite{BK}*{Thm 2.6.1}.  To put it
  another way, replacing the vertex condition~\eqref{quasi-conditions}
  at $B$ with Dirichlet removes all dependence on the quasi-momenta
  $\veck$.  Similarly, the quasi-NK conditions could be equivalently
  imposed at the vertex $A$, where replacing them with Dirichlet also
  removes all dependence on $\veck$.
\end{proof}

\begin{lemma}
  \label{lem:equilateral_case}
  The first eigenvalue $\lambda_1$ of $-\Delta$ on $\Gamma^B$ is
  simple.  If $\ell_1 = \ell_2 = \ell_3 = \ell_4$, the eigenfunction
  corresponding to $\lambda_1$ is identical on these four edges,
  $\phi_1 \equiv \phi_2 \equiv \phi_3 \equiv \phi_4$, and non-zero
  except at $B$.
\end{lemma}

\begin{proof}
  The proof is identical for $\Gamma^B_1$ and $\Gamma^B_2$.
  Simplicity of the eigenvalue follows from general variational
  principles \cite{Kur_lmp19} (or can be deduced from the secular
  equation for the corresponding graphs, see also the proof of
  Proposition~\ref{prop:quant_edge_condition} below).  The first
  eigenfunction is known to be positive, except where a Dirichlet
  condition is enforced, for a large family of vertex conditions
  \cite{Kur_lmp19}.  Symmetry can be deduced by, for example,
  restricting $-\Delta$ to the symmetric subspace of the operator's
  domain \cite{BanBerJoyLiu_prep17}, observing that the first
  eigenfunction of the restricted operator is positive and concluding
  that it corresponds to a positive eigenfunction of the full operator
  and therefore must be the ground state.
\end{proof}

\begin{proof}[Proof of Theorem \ref{main}]
  Since by Lemma~\ref{lem:unitary_equivalence} $\Gamma^A$ and
  $\Gamma^B$ are obtained by a rank-1 Dirichlet perturbation from the
  quantum graph $\Gamma^\veck$ for any $\veck$, Theorem~\ref{rank1}
  yields the inequalities
  \begin{equation}
    \label{eq:rank1_B}
    \lambda_1(\Gamma^\veck) \leq \lambda_1(\Gamma^B)
    \leq \lambda_2(\Gamma^\veck),
  \end{equation}
  and
  \begin{equation}
    \label{eq:rank1_A}
    \lambda_1(\Gamma^\veck) \leq \lambda_1(\Gamma^A)
    \leq \lambda_2(\Gamma^\veck),
  \end{equation}
  which hold of all $\veck \in \mathbb{T}^3$.
  Adding the result of Lemma~\ref{GammaA-B}, we get
  \begin{equation}
    \label{eq:gap_estimate}
    \lambda_1(\Gamma^\veck) \leq \lambda_1(\Gamma^B)
    < \lambda_1(\Gamma^A) \leq \lambda_2(\Gamma^\veck),
  \end{equation}
  obtaining part (a) of Theorem~\ref{main}.

  We will now show that the first inequality in
  \eqref{eq:gap_estimate} turns into equality
  \begin{equation}
    \label{eq:band1equality}
    \lambda_1(\Gamma^\veck)=\lambda_1(\Gamma^B)  
  \end{equation}
  for $k$ in a
  one-dimensional curve $\gamma$ in $\mathbb{T}^3$.
  
  Let $\varphi$ be the $\lambda_1(\Gamma^B)$-eigenfunction of
  $\Gamma^B$.  By Theorem~\ref{rank1},
  equality~\eqref{eq:band1equality} holds if and only if $\varphi$ is
  also an eigenfunction of $\Gamma^\veck$.
  We denote by $\varphi_j$ the restriction of $\varphi$ on $e_j$ for
  $0\leq j\leq 4$. Obviously, $\varphi$ satisfies the first condition
  in \eqref{quasi-conditions} at the vertex $B$. Therefore,
  equality~\eqref{eq:band1equality} holds if and only if
  $\veck=(k_1,k_2,k_3) \in [-\pi,\pi)^3$ is such that 
  \begin{equation}
    \label{Qk-condition}
    \sum_{1 \leq j \leq 3} e^{ik_j}\varphi_j'(B)+\varphi_4'(B)=0.
  \end{equation}
  By Lemma~\ref{lemma-quadrangle}, the set of solutions of
  \eqref{Qk-condition} is a non-trivial algebraic curve of
  co-dimension 2 if 
  \begin{equation}
    \label{Qk-necessary-sufficient}
    2\max_{j} \left|\varphi_j'(B)\right| <
    \sum_{1 \leq j\leq 4} \left|\varphi_j'(B)\right|.
  \end{equation}
  If the lengths $\ell_j$ ($1 \leq j\leq 4$) are
  approximately equal then (by eigenfunction continuity and
  Lemma~\ref{lem:equilateral_case}) all $\left|\varphi_j'(B)\right|$
  are approximately equal and condition
  \eqref{Qk-necessary-sufficient} is satisfied.  This completes the
  proof of part (b).

  Finally, the robustness of the degenerate gap edge under a small
  perturbation of edge lengths or edge potentials follows directly
  from continuity of eigenvalue and eigenfunction data
  \cites{BerKuc_incol12,KucZha_jmp19} and the fact that conditions for
  the degenerate gap edge are inequalities \eqref{eq:gap_condition}
  and \eqref{Qk-necessary-sufficient}.
\end{proof}

With a little extra effort we can provide a \textbf{quantitative
  condition} on the lengths $\ell_j$ to ensure the validity of the
quadrangle inequalities~\eqref{Qk-necessary-sufficient} whenever all
of the derivatives $\varphi_1'(B),\ldots,\varphi_4'(B)$ are not
zero.
\begin{prop}
  \label{prop:quant_edge_condition}
  Let $\rho_0$ be the unique solution in $(2,3)$ to the
  equation $$\rho^2-\frac{\rho^3}{3}=\frac{\pi^2}{24},$$ and assume
  further that
  $$\min\left\{\left(\rho_0\cdot\min_{1 \leq j\leq 4}\ell_j\right), \ell_0\right\} \geq \max_{1 \leq j\leq 4}\ell_j$$
  Then $2\cdot|\varphi_j'(B)|<|\varphi_1'(B)|+|\varphi_2'(B)|+|\varphi_3'(B)|+|\varphi_4'(B)|$ for each $1\leq j\leq 4$. As a consequence, the same conclusion in part (b) of Theorem \ref{main} holds.
\end{prop}
\begin{proof}
  Without loss of generality, assume that $\ell_4 \geq \ell_3 \geq \ell_2 \geq \ell_1$.
  On the edge $e_j$ where $1 \leq j\leq 4$, we write $\varphi_j(x)=\alpha_j \sin(\beta x)$, where $0 \leq x \leq \ell_j$ , $\alpha_j\in \mathbb{R}$ and $\beta=\left(\lambda_1(\Gamma^B)\right)^{1/2}$. Here we identify the vertex $B$ as $x=0$ on each edge $e_j$.  Observe that 
  \begin{equation}
    \label{estimate-lambda}
    0<\beta=\lambda_1(\Gamma^B)^{1/2} \leq \min_{1\leq j \leq 4}\left\{\frac{\pi}{2\ell_0},\frac{\pi}{\ell_j}\right\}=\frac{\pi}{2\ell_0}
  \end{equation}
  This implies that $\beta \ell_j \in (0,\pi/2]$ for each $j$. So $\min_{1 \leq j \leq 4}|\sin(\beta \ell_j)|=\sin(\beta \ell_1)$.
  Moreover, from the fact that $\varphi_j(\ell_j) \neq 0$ and the continuity of $\varphi$ at the vertex $A$, we have
  $$\beta^{-1}\varphi_j'(B)=\alpha_j = \alpha_4 \cdot \frac{\sin(\beta \ell_4)}{\sin(\beta \ell_j)}$$
  Therefore, it is enough to show 
  \begin{equation}
    \label{estimate-sin}
    \sin(\beta \ell_1) \cdot \sum_{i=2}^4 \frac{1}{\sin(\beta \ell_i)}>1
  \end{equation}
  Put $\displaystyle \rho:=\frac{\ell_4}{\ell_1} \in [1,\rho_0]$ then we get 
  $$1-\frac{\pi^2}{24\rho^2}>\frac{\rho}{3}$$
  From \mref{estimate-lambda}, $\beta \ell_1<\frac{\pi \ell_1}{2 \ell_0}\leq \frac{\pi}{2\rho}$ and so it implies 
  \begin{equation}
    \label{estimate-rho}
    1 -  \frac{(\beta \ell_1)^2}{6} >\frac{\ell_4}{3\ell_1}
  \end{equation}
  Since $\sin(\beta \ell_1) \geq \beta \ell_1 -  (\beta \ell_1)^3/6$ and $\sin(\beta \ell_j) \leq (\beta \ell_4)$, \mref{estimate-sin} follows from \mref{estimate-rho}.
\end{proof}

%%%%%%%%%%%%%%%%%%%%%%%%%%%%%% NEW SECTIONS %%%%%%%%%%%%%%%%%%%%%%%%%%%
\section{Discussion}
\label{sec:discussion}

Our Theorem \ref{main} provides a quantum graph counterexample to the
mentioned conjecture at the beginning of the paper, about the
genericity of non-degenerate spectral edges in spectra of
$\mathbb{Z}^d$-periodic quantum graphs, where $d>2$. Note that this
construction can also be modified to provide an example of a
$\mathbb{Z}^d$-discrete graph whose dispersion relation of the
discrete Laplacian operator contains a degenerate band edge. Indeed,
let $\Gamma_d$ be the graph with two vertices such that there are
exactly $d+1$ - edges between them and therefore, its maximal abelian
covering $X_d$ is a $d$-dimensional topological diamond.  One can
write down explicitly the dispersion relation of$X_d$ and then proceed
a similar calculation as in \cite{FilKach} to derive the degeneracy of
the extrema of the band functions.

Our construction of the graphs $\Gamma$ required that the dimension of
the dual torus of quasimomenta $k$ be of dimension at least three. The
same method and proof will still work if we increase the number of
edges connecting the two vertices $A$ and $B$ (at least four edges).
In fact, the entire mechanism of the proof is extremely robust: two
rank-one perturbations that reduce the number of cycles by 3 or more
help create the gaps between conductivity bands, while a continuum of
solutions to an equation similar to \eqref{Qk-condition} will make the
band edge degenerate. The degeneracy curve thus still has a natural
interpretation as the set of possible $(n-1)$-tuples of angles in a
planar $n$-gon with the given edge lengths; here $n$ is the number of
cycles broken by the rank one perturbation.  For $n>3$ the topology of
such objects becomes increasingly complicated.  Their homology groups
were studied by many authors, see \cite{FarSch} and references therein.

A heuristic reason for the degeneracy may be put forward using the
classical idea of Wigner and von~Neumann: a family of complex
Hermitian matrices depending on 3 parameters is expected to have
isolated point degeneracies (where a pair of eigenvalues meet).  This
is what the eigenvalues want to do here, but there are hard bounds
\eqref{eq:rank1_B} and \eqref{eq:rank1_A} from the rank one
perturbations, so the eigenvalues instead accumulate at the bound.

%%%%%%%%%%%%%%%%%%%%%%%%%%%%%%%%%%%%%%%%%%%%%%%%%%%%%%%%%%%%%%%%
\section{Acknowledgments}
The work of the first author was partially supported by NSF
DMS--1815075 grant and the work of the second author was partially
supported by the AMS--Simons Travel grant. Both authors express their
gratitude to Peter Kuchment for introducing them to this exciting
topic and to Lior Alon and Ram Band for many deep discussions.  We
thank an anonymous referee for several improving suggestions.

\end{document}